\documentclass[10pt]{amsart}

\usepackage[margin=1in,headheight=13.6pt]{geometry}

\usepackage{textcmds} 
\usepackage{amsmath}
\usepackage{amsxtra}
\usepackage{amscd}
\usepackage{amsfonts}
\usepackage{amssymb}
\usepackage{eucal}
\usepackage[all]{xy}
\usepackage{graphicx}
\usepackage{comment}
\usepackage{epsfig}
\usepackage{psfrag}
\usepackage{mathrsfs}
\usepackage{amscd}
\usepackage{rotating}
\usepackage{lscape}
\usepackage{amsbsy}
\usepackage{verbatim}
\usepackage{moreverb}
\usepackage{url}
\usepackage{extarrows}

\usepackage{cancel}
\usepackage{bbm}

\usepackage[hidelinks]{hyperref}

\usepackage{tikz-cd}
\usepackage{tikz}
\usetikzlibrary{matrix,arrows,decorations.pathmorphing}

\DeclareMathAlphabet{\mathpzc}{OT1}{pzc}{m}{it}

\newtheorem{thm}{Theorem}[section] 
\newtheorem{cor}[thm]{Corollary}

\newtheorem{qthm}[thm]{Claim}

\newtheorem{rem}{Remark}
\newtheorem{example}{Example}

\numberwithin{equation}{section}

\newcommand{\nc}{\newcommand}
\nc{\renc}{\renewcommand}
\nc{\ssec}{\subsection}
\nc{\sssec}{\subsubsection}
\nc{\on}{\operatorname}

\nc\ol{\overline}
\nc\wt{\widetilde}
\nc\tboxtimes{\wt{\boxtimes}}
\nc\tstar{\wt{\star}}
\nc{\alp}{\alpha}

\nc{\ZZ}{{\mathbb Z}}
\nc{\NN}{{\mathbb N}}
\nc{\OO}{{\mathbb O}}
\renc{\SS}{{\mathbb S}}
\nc{\DD}{{\mathbb D}}
\nc{\GG}{{\mathbb G}}
\nc{\Fq}{{\mathbb F}_q}
\nc{\Fqb}{\ol{{\mathbb F}_q}}
\nc{\Ql}{\ol{{\mathbb Q}_\ell}}
\nc{\id}{\text{id}}
\nc\X{\mathcal X}

\nc{\Hom}{\on{Hom}}
\nc{\Lie}{\on{Lie}}
\nc{\Loc}{\on{Loc}}
\nc{\Pic}{\on{Pic}}
\nc{\Bun}{\on{Bun}}
\nc{\IC}{\on{IC}}
\nc{\Aut}{\on{Aut}}
\nc{\rk}{\on{rk}}
\nc{\Sh}{\on{Sh}}
\nc{\Perv}{\on{Perv}}
\nc{\pos}{{\on{pos}}}
\nc{\Conv}{\on{Conv}}
\nc{\Sph}{\on{Sph}}
\nc{\Sym}{\on{Sym}}
\nc{\BunBb}{\overline{\Bun}_B}
\nc{\BunNb}{\overline{\Bun}_N}
\nc{\BunTb}{\overline{\Bun}_T}
\nc{\BunBbm}{\overline{\Bun}_{B^-}}
\nc{\BunBbel}{\overline{\Bun}_{B,el}}
\nc{\BunBbmel}{\overline{\Bun}_{B^-,el}}
\nc{\Buno}{\overset{o}{\Bun}}
\nc{\BunPb}{{\overline{\Bun}_P}}
\nc{\BunBM}{\Bun_{B(M)}}
\nc{\BunBMb}{\overline{\Bun}_{B(M)}}
\nc{\BunPbw}{{\widetilde{\Bun}_P}}
\nc{\BunBP}{\widetilde{\Bun}_{B,P}}
\nc{\GUb}{\overline{G/U}}
\nc{\GUPb}{\overline{G/U(P)}}

\nc\syminfty{\on{Sym}^{\infty}}
\nc\lal{\ol{\lambda}}
\nc\xl{\ol{x}}
\nc\thl{\ol{\theta}}
\nc\nul{\ol{\nu}}
\nc\mul{\ol{\mu}}
\nc\Sum\Sigma
\nc{\oX}{\overset{\circ}{X}{}}
\nc{\hl}{\overset{\leftarrow}h{}}
\nc{\hr}{\overset{\rightarrow}h{}}
\nc{\M}{{\mathcal M}}
\nc{\N}{{\mathcal N}}
\nc{\F}{{\mathcal F}}
\nc{\D}{{\mathcal D}}
\nc{\Y}{{\mathcal Y}}
\nc{\G}{{\mathcal G}}
\nc{\E}{{\mathcal E}}
\nc{\CalC}{{\mathcal C}}
\nc\Dh{\widehat{\D}}

\nc{\K}{{\mathcal K}}

\nc{\T}{{\mathcal T}}
\nc{\V}{{\mathcal V}}
\renc{\P}{{\mathcal P}}
\nc{\A}{{\mathcal A}}
\nc{\B}{{\mathcal B}}
\nc{\U}{{\mathcal U}}

\nc{\frn}{{\check{\mathfrak u}(P)}}

\nc{\fC}{\mathfrak C}
\nc\f{{\mathfrak f}}

\nc{\qo}{{\mathfrak q}}
\nc{\po}{{\mathfrak p}}
\nc{\s}{{\mathfrak s}}
\nc\w{\text{w}}
\renewcommand{\r}{{\mathfrak r}}

\renewcommand{\mod}{{\on{-}\mathsf{mod}}}

\nc\mathi\iota
\nc\Spec{\on{Spec}}
\nc\Mod{\on{Mod}}
\nc{\tw}{\widetilde{\mathfrak t}}
\nc{\pw}{\widetilde{\mathfrak p}}
\nc{\qw}{\widetilde{\mathfrak q}}
\nc{\jw}{\widetilde j}

\nc{\grb}{\overline{\Gr_{X^{\fset}}}}
\nc{\I}{\mathcal I}
\renewcommand{\i}{\mathfrak i}
\renewcommand{\j}{\mathfrak j}

\nc{\lambdach}{{\check\lambda}}
\nc{\Lambdach}{{\check\Lambda}{}}
\nc{\much}{{\check\mu}}
\nc{\omegach}{{\check\omega}}
\nc{\nuch}{{\check\nu}}
\nc{\etach}{{\check\eta}}
\nc{\alphach}{{\check\alpha}}

\nc{\rhoch}{{\check\rho}}

\nc{\Hb}{\overline{\H}}

\nc{\BA}{{\mathbb{A}}}
\nc{\BB}{\mathbb{B}}
\nc{\BC}{{\mathbb{C}}}
\nc{\BD}{{\mathbb{D}}}
\nc{\BE}{{\mathbb{E}}}
\nc{\BF}{{\mathbb{F}}}
\nc{\BG}{{\mathbb{G}}}
\nc{\BH}{{\mathbb{H}}}
\nc{\BI}{{\mathbb{I}}}
\nc{\BM}{{\mathbb{M}}}
\nc{\BN}{{\mathbb{N}}}
\nc{\BO}{{\mathbb{O}}}
\nc{\BP}{{\mathbb{P}}}
\nc{\BQ}{{\mathbb{Q}}}
\nc{\BR}{{\mathbb{R}}}
\nc{\BS}{{\mathbb{S}}}
\nc{\BT}{{\mathbb{T}}}
\nc{\BV}{{\mathbb{V}}}
\nc{\BZ}{{\mathbb{Z}}}

\nc{\bbone}{\mathbbm{1}}
\nc{\bbA}{{\mathbb{A}}}
\nc{\bbB}{\mathbb{B}}
\nc{\bbC}{{\mathbb{C}}}
\nc{\bbD}{{\mathbb{D}}}
\nc{\bbE}{{\mathbb{E}}}
\nc{\bbF}{{\mathbb{F}}}
\nc{\bbG}{{\mathbb{G}}}
\nc{\bbH}{{\mathbb{H}}}
\nc{\bbI}{{\mathbb{I}}}
\nc{\bbL}{{\mathbb{L}}}
\nc{\bbM}{{\mathbb{M}}}
\nc{\bbN}{{\mathbb{N}}}
\nc{\bbO}{{\mathbb{O}}}
\nc{\bbP}{{\mathbb{P}}}
\nc{\bbQ}{{\mathbb{Q}}}
\nc{\bbR}{{\mathbb{R}}}
\nc{\bbS}{{\mathbb{S}}}
\nc{\bbT}{{\mathbb{T}}}
\nc{\bbU}{{\mathbb{U}}}
\nc{\bbV}{{\mathbb{V}}}
\nc{\bbW}{{\mathbb{W}}}
\nc{\bbX}{{\mathbb{X}}}
\nc{\bbY}{{\mathbb{Y}}}
\nc{\bbZ}{{\mathbb{Z}}}

\nc{\CA}{{\mathcal{A}}}
\nc{\CB}{{\mathcal{B}}}

\nc{\CE}{{\mathcal{E}}}
\nc{\CF}{{\mathcal{F}}}
\nc{\CH}{{\mathcal{H}}}

\nc{\CL}{{\mathcal{L}}}
\nc{\CC}{{\mathcal{C}}}
\nc{\CG}{{\mathcal{G}}}
\nc{\CM}{{\mathcal{M}}}
\nc{\CN}{{\mathcal{N}}}
\nc{\CK}{{\mathcal{K}}}
\nc{\CO}{{\mathcal{O}}}
\nc{\CP}{{\mathcal{P}}}
\nc{\CQ}{{\mathcal{Q}}}
\nc{\CR}{{\mathcal{R}}}
\nc{\CS}{{\mathcal{S}}}
\nc{\CU}{{\mathcal{U}}}
\nc{\CV}{{\mathcal{V}}}
\nc{\CW}{{\mathcal{W}}}
\nc{\CX}{{\mathcal{X}}}
\nc{\CY}{{\mathcal{Y}}}
\nc{\CZ}{{\mathcal{Z}}}
\nc{\CI}{{\mathcal{I}}}

\nc{\csM}{{\check{\mathcal A}}{}}
\nc{\oM}{{\overset{\circ}{\mathcal M}}{}}
\nc{\obM}{{\overset{\circ}{\mathbf M}}{}}
\nc{\oCA}{{\overset{\circ}{\mathcal A}}{}}
\nc{\obA}{{\overset{\circ}{\mathbf A}}{}}
\nc{\ooM}{{\overset{\circ}{M}}{}}
\nc{\osM}{{\overset{\circ}{\mathsf M}}{}}
\nc{\vM}{{\overset{\bullet}{\mathcal M}}{}}
\nc{\nM}{{\underset{\bullet}{\mathcal M}}{}}
\nc{\oD}{{\overset{\circ}{\mathcal D}}{}}
\nc{\obC}{{\overset{\circ}{\mathbf C}}{}}
\nc{\obD}{{\overset{\circ}{\mathbf D}}{}}
\nc{\oA}{{\overset{\circ}{\mathbb A}}{}}
\nc{\op}{{\overset{\bullet}{\mathbf p}}{}}
\nc{\oU}{{\overset{\bullet}{\mathcal U}}{}}
\nc{\oZ}{{\overset{\circ}{\mathcal Z}}{}}
\nc{\ofZ}{{\overset{\circ}{\mathfrak Z}}{}}
\nc{\oF}{{\overset{\circ}{\fF}}}

\nc{\fa}{{\mathfrak{a}}}
\nc{\fb}{{\mathfrak{b}}}
\nc{\fc}{{\mathfrak{c}}}
\nc{\fd}{{\mathfrak{d}}}
\nc{\ff}{{\mathfrak{f}}}
\nc{\fg}{{\mathfrak{g}}}
\nc{\fgl}{{\mathfrak{gl}}}
\nc{\fh}{{\mathfrak{h}}}
\nc{\fj}{{\mathfrak{j}}}
\nc{\fl}{{\mathfrak{l}}}
\nc{\fm}{{\mathfrak{m}}}
\nc{\fn}{{\mathfrak{n}}}
\nc{\fu}{{\mathfrak{u}}}
\nc{\fp}{{\mathfrak{p}}}
\nc{\fr}{{\mathfrak{r}}}
\nc{\fs}{{\mathfrak{s}}}
\nc{\ft}{{\mathfrak{t}}}
\nc{\fz}{{\mathfrak{z}}}
\nc{\fsl}{{\mathfrak{sl}}}
\nc{\hsl}{{\widehat{\mathfrak{sl}}}}
\nc{\hgl}{{\widehat{\mathfrak{gl}}}}
\nc{\hg}{{\widehat{\mathfrak{g}}}}
\nc{\chg}{{\widehat{\mathfrak{g}}}{}^\vee}
\nc{\hn}{{\widehat{\mathfrak{n}}}}
\nc{\chn}{{\widehat{\mathfrak{n}}}{}^\vee}

\nc{\fA}{{\mathfrak{A}}}
\nc{\fB}{{\mathfrak{B}}}
\nc{\fD}{{\mathfrak{D}}}

\nc{\fF}{{\mathfrak{F}}}
\nc{\fG}{{\mathfrak{G}}}
\nc{\fK}{{\mathfrak{K}}}
\nc{\fL}{{\mathfrak{L}}}
\nc{\fM}{{\mathfrak{M}}}
\nc{\fN}{{\mathfrak{N}}}
\nc{\fP}{{\mathfrak{P}}}
\nc{\fU}{{\mathfrak{U}}}
\nc{\fV}{{\mathfrak{V}}}
\nc{\fX}{{\mathfrak{X}}}
\nc{\fY}{{\mathfrak{Y}}}
\nc{\fZ}{{\mathfrak{Z}}}

\nc{\bb}{{\mathbf{b}}}
\nc{\bc}{{\mathbf{c}}}
\nc{\bd}{{\mathbf{d}}}
\nc{\bbf}{{\mathbf{f}}}
\nc{\be}{{\mathbf{e}}}
\nc{\bg}{{\mathbf{g}}}
\nc{\bj}{{\mathbf{j}}}
\nc{\bn}{{\mathbf{n}}}
\nc{\bo}{{\mathbf{o}}}
\nc{\bp}{{\mathbf{p}}}
\nc{\bq}{{\mathbf{q}}}
\nc{\bt}{{\mathbf{t}}}
\nc{\bu}{{\mathbf{u}}}
\nc{\bv}{{\mathbf{v}}}
\nc{\bx}{{\mathbf{x}}}
\nc{\bs}{{\mathbf{s}}}
\nc{\by}{{\mathbf{y}}}
\nc{\bw}{{\mathbf{w}}}
\nc{\bA}{{\mathbf{A}}}
\nc{\bK}{{\mathbf{K}}}
\nc{\bB}{{\mathbf{B}}}
\nc{\bC}{{\mathbf{C}}}
\nc{\bG}{{\mathbf{G}}}
\nc{\bD}{{\mathbf{D}}}
\nc{\bH}{{\mathbf{H}}}
\nc{\bM}{{\mathbf{M}}}
\nc{\bN}{{\mathbf{N}}}
\nc{\bO}{{\mathbf{O}}}
\nc{\bT}{{\mathbf{T}}}
\nc{\bV}{{\mathbf{V}}}
\nc{\bW}{{\mathbf{W}}}
\nc{\bX}{{\mathbf{X}}}
\nc{\bZ}{{\mathbf{Z}}}
\nc{\bS}{{\mathbf{S}}}

\nc{\sA}{{\mathsf{A}}}
\nc{\sB}{{\mathsf{B}}}
\nc{\sC}{{\mathsf{C}}}
\nc{\sD}{{\mathsf{D}}}
\nc{\sF}{{\mathsf{F}}}
\nc{\sG}{{\mathsf{G}}}
\nc{\sK}{{\mathsf{K}}}
\nc{\sM}{{\mathsf{M}}}
\nc{\sO}{{\mathsf{O}}}
\nc{\sW}{{\mathsf{W}}}
\nc{\sQ}{{\mathsf{Q}}}
\nc{\sP}{{\mathsf{P}}}
\nc{\sV}{{\mathsf{V}}}
\nc{\sS}{{\mathsf{S}}}
\nc{\sT}{{\mathsf{T}}}
\nc{\sZ}{{\mathsf{Z}}}
\nc{\sfp}{{\mathsf{p}}}
\nc{\sll}{{\mathsf{l}}}
\nc{\sr}{{\mathsf{r}}}
\nc{\bk}{{\mathsf{k}}}
\nc{\sg}{{\mathsf{g}}}

\nc{\sff}{{\mathsf{f}}}
\nc{\sfb}{{\mathsf{b}}}
\nc{\sfc}{{\mathsf{c}}}
\nc{\sd}{{\mathsf{d}}}
\nc{\se}{{\mathsf{e}}}

\nc{\BK}{{\bar{K}}}

\nc{\tA}{{\widetilde{\mathbf{A}}}}
\nc{\tB}{{\widetilde{\mathcal{B}}}}
\nc{\tg}{{\widetilde{\mathfrak{g}}}}
\nc{\tG}{{\widetilde{G}}}
\nc{\TM}{{\widetilde{\mathbb{M}}}{}}
\nc{\tO}{{\widetilde{\mathsf{O}}}{}}
\nc{\tU}{{\widetilde{\mathfrak{U}}}{}}
\nc{\TZ}{{\tilde{Z}}}
\nc{\tx}{{\tilde{x}}}
\nc{\tbv}{{\tilde{\bv}}}
\nc{\tfP}{{\widetilde{\mathfrak{P}}}{}}
\nc{\tz}{{\tilde{\zeta}}}
\nc{\tmu}{{\tilde{\mu}}}

\nc{\urho}{\underline{\rho}}
\nc{\uB}{\underline{B}}
\nc{\uC}{{\underline{\mathbb{C}}}}
\nc{\ui}{\underline{i}}
\nc{\uj}{\underline{j}}
\nc{\ofP}{{\overline{\mathfrak{P}}}}
\nc{\oB}{{\overline{\mathcal{B}}}}
\nc{\og}{{\overline{\mathfrak{g}}}}
\nc{\oI}{{\overline{I}}}

\nc{\eps}{\varepsilon}
\nc{\hrho}{{\hat{\rho}}}

\nc{\one}{{\mathbf{1}}}
\nc{\two}{{\mathbf{t}}}

\nc{\Rep}{{\mathop{\operatorname{\rm Rep}}}}
\nc{\Tot}{{\mathop{\operatorname{\rm Tot}}}}
\nc{\Ker}{{\mathop{\operatorname{\rm Ker}}}}
\nc{\Hilb}{{\mathop{\operatorname{\rm Hilb}}}}
\nc{\Ext}{{\mathop{\operatorname{\rm Ext}}}}
\nc{\CHom}{{\mathop{\operatorname{{\mathcal{H}}\it om}}}}
\nc{\GL}{{\mathop{\operatorname{\rm GL}}}}
\nc{\gr}{{\mathop{\operatorname{\rm gr}}}}
\nc{\Id}{{\mathop{\operatorname{\rm Id}}}}
\nc{\de}{{\mathop{\operatorname{\rm def}}}}
\nc{\length}{{\mathop{\operatorname{\rm length}}}}
\nc{\supp}{{\mathop{\operatorname{\rm supp}}}}

\nc{\Cliff}{{\mathsf{Cliff}}}
\nc{\Fl}{\on{Fl}}
\nc{\Fib}{{\mathsf{Fib}}}
\nc{\Coh}{{\on{Coh}}}
\nc{\QCoh}{{\on{QCoh}}}
\nc{\IndCoh}{{\on{IndCoh}}}
\nc{\FCoh}{{\mathsf{FCoh}}}

\nc{\reg}{{\text{\rm reg}}}

\nc{\cplus}{{\mathbf{C}_+}}
\nc{\cminus}{{\mathbf{C}_-}}
\nc{\cthree}{{\mathbf{C}_*}}
\nc{\Qbar}{{\bar{Q}}}
\nc\Eis{\on{Eis}}
\nc\Eisb{\ol\Eis{}}
\nc\Eisr{\on{Eis}^{rat}{}}
\nc\wh{\widehat}
\nc{\Def}{\on{Def_{\check{\fb}}(E)}}
\nc{\barZ}{\overline{Z}{}}
\nc{\barbarZ}{\overline{\barZ}{}}
\nc{\barpi}{\overline\pi}
\nc{\barbarpi}{\overline\barpi}
\nc{\barpip}{\overline\pi{}^+}
\nc{\barpim}{\overline\pi{}^-}

\nc{\fq}{\mathfrak q}

\nc{\fqb}{\ol{\fq}{}}
\nc{\fpb}{\ol{\fp}{}}
\nc{\fpr}{{\fp^{rat}}{}}
\nc{\fqr}{{\fq^{rat}}{}}

\nc{\hattimes}{\wh\otimes}

\nc{\bh}{{\bar{h}}}
\nc{\bOmega}{{\overline{\Omega(\check \fn)}}}

\nc{\seq}[1]{\stackrel{#1}{\sim}}

\nc{\cT}{{\check{T}}}
\nc{\cG}{{\check{G}}}
\nc{\cM}{{\check{M}}}
\nc{\cB}{{\check{B}}}
\nc{\cP}{{\check{P}}}

\nc{\ct}{{\check{\mathfrak t}}}
\nc{\cg}{{\check{\fg}}}
\nc{\cb}{{\check{\fb}}}
\nc{\cn}{{\check{\fn}}}
\nc{\cp}{{\check{\fp}}}
\nc{\cm}{{\check{\fm}}}

\nc{\cLambda}{{\check\Lambda}}

\nc{\cla}{{\check\lambda}}
\nc{\cmu}{{\check\mu}}
\nc{\cnu}{{\check\nu}}
\nc{\ceta}{{\check\eta}}

\nc{\DefbE}{{\on{Def}_{\cB}(E_\cT)}}

\nc{\imathb}{{\ol{\imath}}}
\nc{\rlr}{\overset{\longrightarrow}{\underset{\longrightarrow}\longleftarrow}}

\nc{\oBun}{\overset{\circ}\Bun}
\nc{\BunBbb}{\ol{\ol{Bun}}_B}
\nc{\BunBr}{\Bun_B^{rat}}
\nc{\BunBrsg}{\Bun_B^{rat,\on{s.g.}}}
\nc{\BunBrp}{\Bun_B^{rat,polar}}
\nc{\BunBrpbg}{\Bun_B^{rat,polar,\on{b.g.}}}
\nc{\BunBrpsg}{\Bun_B^{rat,polar,\on{s.g.}}}
\nc{\BunTrp}{\Bun_T^{rat,polar}}
\nc{\BunTrpbg}{\Bun_T^{rat,polar,\on{b.g.}}}
\nc{\BunTrpsg}{\Bun_T^{rat,polar,\on{s.g.}}}
\nc{\BunNr}{\Bun_N^{rat}}
\nc{\BunNre}{\Bun_N^{enh,rat}}
\nc{\BunTr}{\Bun_T^{rat}}
\nc{\Vect}{\on{Vect}}
\nc{\Whit}{\on{Whit}}
\nc{\bTb}{\ol{\on{CT}}}

\nc{\bTr}{\on{CT}^{rat}{}}
\nc\jmathr{\jmath^{rat}{}}
\nc{\ux}{\underline{x}}
\nc{\clambda}{{\check\lambda}}
\nc{\calpha}{{\check\alpha}}

\nc{\inftyGrpd}{{\mathsf{Grpd}_\infty}}

\nc{\fset}{\mathsf{fSet}}
\nc{\LocSysG}{\LocSys_{\cG}}
\nc{\Sing}{{\on{Sing}}}
\nc{\dr}{{\on{dR}}}
\nc{\Ind}{\on{Ind}}
\nc{\Sat}{\on{Sat}}
\nc{\Ho}{\on{Ho}}
\nc{\Res}{\on{Res}}
\nc{\sotimes}{\overset{!}\otimes}
\nc{\mmod}{{\on{-}}{\mathbf{mod}}}
\nc{\Maps}{\on{Maps}}
\nc{\CMaps}{{\mathcal Maps}}
\nc{\bMaps}{{\mathbf{Maps}}}

\nc{\dgSch}{\on{DGSch}}
\nc{\dgindSch}{\on{DGindSch}}
\nc{\indSch}{\on{indSch}}
\nc{\Sch}{\mathsf{Sch}}
\nc{\affdgSch}{\on{DGSch}^{\on{aff}}}
\nc{\affSch}{\on{Sch}^{\on{aff}}}

\nc{\Groupoids}{\on{Grpd}}
\nc{\inftypic}{\infty\on{-PicGrpd}}
\nc{\inftyCat}{{\mathsf{Cat}_{\infty}}}
\nc{\MoninftyCat}{\infty\on{-Cat}^{Mon}}
\nc{\SymMoninftyCat}{\infty\on{-Cat}^{\on{SymMon}}}
\nc{\SymMonStinftyCat}{\on{DGCat}^{\on{SymMon}}}
\nc{\MonStinftyCat}{\on{DGCat}^{Mon}}
\nc{\inftystack}{\on{Stk}}
\nc{\inftystackalg}{Stk^{1\text{-}alg}}
\nc{\inftyprestack}{\on{PreStk}}
\nc{\inftydgnearstack}{\on{NearStk}}
\nc{\inftydgstack}{\on{Stk}}
\nc{\inftydgstackalg}{DGStk^{1\text{-}alg}}
\nc{\inftydgprestack}{\on{PreStk}}

\nc{\HC}{\CH\bC}

\nc{\csupp}{\supp}
\nc{\Arth}{\on{Arth}}
\nc{\ArthG}{{\on{Arth}_\cG}}
\nc{\ul}{\underline}

\nc{\Z}{\mathcal{Z}}

\nc{\calN}{\N}
\nc{\calW}{\mathcal{W}}
\nc{\calF}{\mathcal{F}}
\nc{\calH}{\mathcal{H}}
\nc{\calO}{\mathcal{O}}
\nc{\calK}{\mathcal{K}}

\nc{\Ran}{\mathsf{Ran}}
\nc{\Jets}{\on{Jets}}
\nc{\act}{\mathsf{act}}
\nc{\Av}{\mathsf{Av}}
\nc{\Ad}{\on{Ad}}
\nc{\BGRan}{BG_{\Ran}}
\nc{\colim}{\on{colim}}
\nc{\codim}{\on{codim}}
\nc{\cpt}{{\on{cpt}}}
\nc{\dR}{{\on{dR}}}
\nc{\DGCat}{\mathsf{DGCat}}
\nc{\DGCatcont}{\on{DGCat}_{cont}}
\nc{\glob}{{\on{glob}}}
\nc{\loc}{{\on{loc}}}

\renewcommand{\op}{{\on{op}}}
\nc{\pt}{{\on{pt}}}
\nc{\PreStk}{{\mathsf{PreStk}}}
\nc{\Cat}{{\mathsf{Cat}}}

\nc{\ShvCat}{{\mathsf{ShvCat}}}

\nc{\restr}[2]{\left. #1 \right |_{#2}}
\nc{\uprestr}[2]{\left. #1 \right |^{#2}}

\nc{\bLoc}{{\mathbf{Loc}}}
\nc{\bGamma}{{\mathbf{\Gamma}}}
\nc{\bLocA}{\mathbf{Loc}^\A}
\nc{\bGammaA}{\mathbf{\Gamma}^\A}
\nc{\bLocB}{\mathbf{Loc}^\B}
\nc{\bGammaB}{\mathbf{\Gamma}^\B}
\nc{\bLocH}{\mathbf{Loc}^\H}
\nc{\bGammaH}{\mathbf{\Gamma}^\H}
\nc{\gen}{\mathsf{gen}}

\nc{\hto}{\hookrightarrow}

\nc{\ext}{\mathsf{ext}}

\nc{\ev}{\mathsf{ev}}
\nc{\rat}{\mathsf{rat}}

\nc{\usotimes}[1]{\underset{#1}{\otimes}}
\nc{\ustimes}[1]{\underset{#1}{\times}}
\nc{\uscolim}[1]{\underset{#1}{\colim}}

\nc{\ch}{{\mathfrak{ch}}}

\renc{\fD}{{\Dmod}}
\nc{\fH}{{\mathfrak{H}}}

\nc{\p}{{\mathfrak{p}}}
\renc{\r}{{\mathfrak{r}}}

\nc{\xto}{\xrightarrow}

\renc{\sec}{\section}
\nc{\enh}{\mathsf{enh}}

\renc{\gen}{\mathsf{gen}}
\nc{\BunGBgen}{\Bun_G^{B-\gen}}
\nc{\BunGHgen}{\Bun_G^{H-\gen}}
\nc{\BunGNgen}{\Bun_G^{N-\gen}}

\nc{\Fun}{\mathsf{Fun}}
\nc{\End}{\mathsf{End}}

\nc{\lr}{\xymatrix{ \ar@<-0.4ex>[r] \ar@<.5ex>[l]  & } }
\nc{\rr}{\xymatrix{ \ar@<-0.2ex>[r] \ar@<.7ex>[r]  & } }
\nc{\rrr}{\xymatrix{ \ar@<.0ex>[r] \ar@<.7ex>[r] \ar@<-0.7ex>[r] & } }

\nc{\Stab}{\mathsf{Stab}}
\nc{\Orb}{\mathsf{Orb}}

\renc{\exp}{\mathit{exp}}

\renc{\q}{\mathfrak{q}}

\nc{\virg}[1]{``#1"}

\renc{\bold}[1]{\boldsymbol{#1}}

\nc{\bigt}[1]{\big( #1 \big) }
\nc{\Bigt}[1]{\Big( #1 \Big) }

\nc{\extwhit}{{\CW h}(G,\mathsf{ext})}

\nc{\footcite}{\footnote}

\nc{\GA}{{G(\AA)}}
\nc{\GO}{{G(\OO)}}

\nc{\Shv}{\mathsf{Shv}}
\nc{\inc}{\mathsf{inc}}

\nc{\Par}{\mathsf{Par}}

\renc{\i}{\mathfrak{i}}

\nc{\NA}{N(\AA)}
\nc{\VA}{V(\AA)}

\nc{\Glue}{\mathsf{Glue}}
\nc{\laxlim}{\text{laxlim}}

\nc{\FT}{\mathsf{FT}}

\nc{\out}{\mathsf{out}}

\nc{\hol}{\mathsf{hol}}
\nc{\Hol}{\on{Hol}}
\nc{\add}{\mathsf{add}}

\nc{\sto}{\rightsquigarrow}
\nc{\squigto}{\rightsquigarrow}

\nc{\fW}{\mathfrak{W}}

\nc{\vrho}{\varrho}

\nc{\counit}{\mathsf{counit}}
\nc{\unit}{\mathsf{unit}}
\nc{\corr}{\mathsf{corr}}
\nc{\Corr}{\mathsf{Corr}}

\nc{\IndSch}{\mathsf{IndSch}}
\nc{\Tate}{{\mathsf{Tate}}}

\nc{\surjto}{\twoheadrightarrow}

\renc{\j}{\mathfrak{j}}

\nc{\J}{\mathcal{J}}

\nc{\pro}{\mathsf{pro}}
\nc{\fty}{\mathsf{ft}}
\nc{\Pro}{\mathsf{Pro}}

\nc{\coact}{\mathsf{coact}}
\nc{\aff}{\mathsf{aff}}

\nc{\Nilp}{\on{Nilp}}

\nc{\Gch}{{\check{G}}}
\nc{\Pch}{{\check{P}}}
\nc{\Mch}{{\check{M}}}
\nc{\Qch}{{\check{Q}}}

\nc{\LL}{\mathbb{L}}

\nc{\LS}{{\on{LS}}}

\nc{\x}{\varkappa} 

\nc{\Otimes}{\boldsymbol{\otimes}}
\nc{\Times}{\boldsymbol{\times}}
\nc{\flip}{\text{<}}

\nc{\coeffRan}{\mathsf{coeff}^{\Ran}}

\nc{\Ha}{H(\sA)}
\nc{\Groups}{\mathsf{Groups}}

\nc{\Groth}{\mathsf{Groth}}

\nc{\rlto}{\rightleftarrows}

\nc{\DGCatRan}{\ShvCatCrys(\Ran)}

\nc{\longto}{\longrightarrow}

\renc{\Jets}{\mathsf{Jets}}
\nc{\mer}{\mathsf{mer}}

\nc{\W}{\mathcal{W}}

\nc{\Sect}{\mathsf{Sect}}
\renc{\Maps}{\mathsf{Maps}}

\renc{\bf}{\mathbf{f}}

\nc{\y}{\mathtt{y}}
\renc{\x}{\mathtt{x}}

\nc{\un}{{\it un}}
\nc{\indep}{\mathsf{indep}}
\nc{\CoAlg}{\mathsf{CoAlg}}

\nc{\coeff}{\mathsf{coeff}}

\nc{\R}{\mathcal{R}}

\renc{\hat}{\widehat}

\nc{\TK}{T(\mathsf{K})} 
\nc{\TtKK}{\Tt(\mathpzc{K})} 
\nc{\TtK}{\Tt(\mathsf{K})} 
\nc{\KK}{\mathpzc{K}}

\nc{\Dmod}{\mathfrak{D}}

\nc{\curs}[1]{\mathpzc{#1}}

\nc{\Bshv}{\bold{\B}}
\nc{\Bind}{\H_{\indep}}
\nc{\BRan}{\H_{\Ran}}
\nc{\ARan}{\A_{\Ran}}
\nc{\Aind}{\A_{\indep}}

\nc{\GrRan}{\Gr}
\nc{\Gr}{\mathsf{Gr}}
\nc{\GrGRan}{\Gr_{G}}
\nc{\GrGind}{\Gr_{G}^{\indep}}
\nc{\Grind}[1]{\Gr_{#1}^{\indep} }
\nc{\GrGdom}{\curs{Gr}_G}

\nc{\GMapsRan}[1]{\mathsf{GMaps}(X,{#1})}
\nc{\GSectRan}[1]{\mathsf{GSect}({#1}/X)}
\nc{\GMapsind}[1]{\mathsf{GMaps}(X,{#1})^\indep}
\nc{\GSectind}[1]{\mathsf{GSect}({#1}/X)^\indep}
\nc{\GMapsdom}[1]{\curs{GMaps}(X,{#1})}
\nc{\GSectdom}[1]{\curs{GSect}({#1}/X)}

\nc{\chind}{\ch^{\indep}}
\nc{\chdom}{\curs{ch}}

\nc{\QSect}[1]{\curs{QSect}(#1/X)} 
\nc{\QMaps}[1]{\curs{QMaps}(X,#1)} 

\nc{\Zar}{\mathit{Zar}}

\nc{\loccit}{\textit{loc.$\,$cit.}}
\nc{\Crys}{\on{Crys}}
\nc{\ShvCatCrys}{\ShvCat^{\Crys}}

\nc{\BPE}{{\BP E}}
\nc{\BVE}{{\BV E}}
\nc{\BBE}{{\BB E}}

\nc{\Wh}{{{\CW}h}}

\nc{\ChiralCat}{\mathsf{ChiralCat}}
\nc{\RRep}{\mathfrak{R}ep}
\nc{\SSph}{\mathfrak{S}ph}

\nc{\tto}{\twoheadrightarrow}
\nc{\disj}{{\mathsf{disj}}}

\nc{\C}{\CC}
\nc{\Tch}{{\check{T}}}

\nc{\good}{\mathsf{good}}
\nc{\triv}{\mathsf{triv}}

\nc{\Alg}{\mathsf{Alg}}
\nc{\CAlg}{\mathsf{CAlg}}
\nc{\Spread}{\mathsf{Spread}}
\nc{\Dom}{\mathsf{Dom}}

\nc{\Jac}{\on{Jac}}
\renc{\CD}[1]{{#1}^{\on{CD}}}

\nc{\String}{\on{String}}

\renc{\min}{{\mathit{min}}}

\nc{\rrep}{\on-\!\mathbf{rep}}

\nc{\WWh}{\mathfrak{W}h}
\nc{\Grpd}{\mathsf{Grpd}}

\nc{\timesdisj}{\overset{\circ}\times}

\renc{\NA}{N(\sA)}
\nc{\chiral}{\mathsf{chiral}}

\nc{\Hopf}{\mathsf{Hopf}}

\nc{\heart}{\heartsuit}
\nc{\kk}{\mathbbm{k}} 

\nc{\HHom}{\CH{om}} 
\nc{\Cone}{\on{Cone}}

\nc{\EE}{\mathbb{E}}
\renc{\HC}{{\on{HC}}}
\nc{\HH}{{\on{HH}}}
\nc{\even}{{\on{even}}}
\nc{\SingSupp}{\on{SingSupp}}
\nc{\Supp}{\on{Supp}}

\nc{\temp}{{\mathit{temp}}}
\nc{\geom}{{\mathit{geom}}}
\nc{\ren}{{\mathit{ren}}}
\nc{\naive}{{\mathit{naive}}}
\nc{\conaive}{{\mathit{conaive}}}
\nc{\spec}{\mathit{spec}}

\nc{\gch}{\mathfrak{\check{g}}}

\nc{\Hecke}{\on{Hecke}}

\nc{\LSGch}{{\LS_\Gch}}

\nc{\Hsx}[2]{\H_{{#1} \leftarrow {#2}}}
\nc{\Hdx}[2]{\H_{{#1} \to {#2}}}
\nc{\Hcorr}[3]{ \H_{{#1} \leftarrow {#2} \to {#3}} }
\nc{\Hopcorr}[3]{ \H_{{#1} \to {#2} \leftto {#3}} }

\nc{\ICohsx}[2]{\ICohW_{{#1} \leftarrow {#2}}}
\nc{\ICohdx}[2]{\ICohW_{{#1} \to {#2}}}
\nc{\ICohcorr}[3]{ \ICohW_{{#1} \leftarrow {#2} \to {#3}} }
\nc{\ICohopcorr}[3]{ \ICohW_{{#1} \to {#2} \leftto {#3}} }

\nc{\QCohsx}[2]{\QCohW_{{#1} \leftarrow {#2}}}
\nc{\QCohdx}[2]{\QCohW_{{#1} \to {#2}}}
\nc{\QCohcorr}[3]{ \QCohW_{{#1} \leftarrow {#2} \to {#3}} }
\nc{\QCohopcorr}[3]{ \QCohW_{{#1} \to {#2} \leftto {#3}} }

\renc{\AA}{\bbA}
\nc{\Asx}[2]{\AA_{{#1} \leftarrow {#2}}}
\nc{\Adx}[2]{\AA_{{#1} \to {#2}}}
\nc{\Acorr}[3]{ \AA_{{#1} \leftarrow {#2} \to {#3}} }
\nc{\Aopcorr}[3]{ \AA_{{#1} \to {#2} \leftto {#3}} }

\nc{\Bsx}[2]{\B_{{#1} \leftarrow {#2}}}
\nc{\Bdx}[2]{\B_{{#1} \to {#2}}}
\nc{\Bcorr}[3]{ \B_{{#1} \leftarrow {#2} \to {#3}} }
\nc{\Bopcorr}[3]{ \B_{{#1} \to {#2} \leftto {#3}} }

\nc{\ICohzero}[3]{\ICoh_0 \bigt{#1 \times_{{#2}_\dR} #3}}
\nc{\IndCohzero}{\ICohzero}

\nc{\form}[3]{#1 \times_{{#2}_\dR} #3 }

\nc{\ind}{{\mathsf{ind}}}
\nc{\oblv}{{\mathsf{oblv}}}

\nc{\Aff}{\mathsf{Aff}}
\nc{\dgAff}{\Aff}

\nc{\deloop}{\mathsf{deloop}}
\renc{\loop}{\mathsf{loop}}

\nc{\coev}{\mathsf{coev}}

\nc{\bE}{\mathbf{E}}

\nc{\ShvCatH}{{\ShvCat^{\bbH}}}
\nc{\ShvCatQW}{\ShvCat^{\QCohW}}

\nc{\bbimod}{\on{-}\mathbf{bimod}}

\nc{\Tw}{\mathsf{Tw}}
\nc{\Arr}{\mathsf{Arr}}

\nc{\bDelta}{\bold\Delta}
\nc{\BiCat}{\mathsf{BiCat}}
\nc{\Seg}{\mathsf{Seg}}
\nc{\Cart}{\mathsf{Cart}}
\nc{\Bimod}{\mathsf{Bimod}}
\nc{\lax}{\mathit{lax}}

\nc{\pr}{\mathsf{pr}}

\nc{\zero}{ \{ 0 \}   }
\nc{\Perf}{\mathsf{Perf}}

\nc{\leftto}{\leftarrow}
\nc{\lto}{\leftto}
\nc{\xlto}[1]{\xleftarrow{#1}}

\nc{\ltemp}{{}^\temp}
\nc{\TwCorr}{\mathsf{TwCorr}}

\nc{\Affover}[1]{{\Aff_{/#1}}}
\nc{\Affoverop}[1]{{( \Affover{#1})^\op}}

\nc{\AffOver}[2]{{(\Aff_{#1})_{/#2}}}

\nc{\AffOverop}[2]{{( \AffOver{#1}{#2})^\op}}

\nc{\aft}{{\mathit{aft}}}

\renc{\vert}{{\mathit{vert}}}
\nc{\horiz}{{\mathit{horiz}}}
\nc{\type}{{\mathit{type}}}
\nc{\adm}{{\mathit{adm}}}

\nc{\g}{\mathfrak{g}}

\nc{\free}{\mathsf{free}}

\nc{\Sform}{{S \times_{S_\dR} S}}
\nc{\Yform}{{\Y \times_{\Y_\dR} \Y}}
\nc{\SdR}{ {S_{\dR}}}

\nc{\laft}{{\mathit{laft}}}

\nc{\Affevcocaft}{\Aff_{\aft}^{< \infty}}
\nc{\Affaftevcoc}{\Aff_{\aft}^{< \infty}}

\nc{\Affevcoclfp}{\Aff_{\lfp}^{< \infty}}

\nc{\Schevcoclfp }{\Sch_{\lfp}^{< \infty}}
\nc{\Schevcocaft}{\Sch_{\aft}^{< \infty}}
\nc{\Schaftevcoc}{\Sch_{\aft}^{< \infty}}

\nc{\Stkevcoc}{\Stk^{< \infty}}
\nc{\Stkevcoclfp}{\Stk_{\lfp}^{< \infty}}
\nc{\Stkperfevcoclfp}{\Stk_{\mathit{perf},\lfp}^{< \infty}}
\nc{\Stkperflfp}{\Stk_{\mathit{perf},\lfp}}
\nc{\Stklfp}{\Stk_{\lfp}}

\nc{\evcoc}{\mathit{e.c.}}

\nc{\ICoh}{\IndCoh}

\nc{\citep}{\cite}

\renc{\H}{\bbH}

\nc{\uno}{\mathbbm{1}}

\nc{\CohBig}{{\Coh^{-\infty}}}

\nc{\Tang}{\mathbb{T}}

\nc{\LieAlg}{\mathsf{LieAlg}}

\nc{\Serre}{{\on{Serre}}}

\nc{\MPreStk}{\mathsf{MPreStk}}

\nc{\all}{{\on{all}}}

\nc{\QCohwedge}{\bbQ^\wedge}
\nc{\ICohwedge}{\bbI^\wedge}
\nc{\ICohW}{\ICohwedge}
\nc{\QCohW}{\QCohwedge}

\nc{\ShvCatA}{\ShvCat^{\AA}}
\nc{\ShvCatB}{{\ShvCat^\B}}

\nc{\naiveto}{{\xto{\naive}}}
\nc{\conaiveto}{{\xto{\conaive}}}

\nc{\strong}{\mathit{strong}}
\nc{\costrong}{\mathit{costrong}}

\nc{\conv}{\mathit{conv}}

\nc{\Q}{\bbQ}

\nc{\bY}{\mathbf{Y}}
\nc{\Loop}{\mathsf{LOOP}}

\nc{\DG}{{\on{DG}}}

\nc{\coind}{\mathsf{coind}}

\nc{\co}{\on{co}}

\nc{\laftdef}{{\mathit{laft-def}}}

\nc{\qsmooth}{{\mathit{qs.smooth}}}
\nc{\smooth}{{\mathit{smooth}}}

\nc{\LKE}{\on{LKE}}
\nc{\RKE}{\on{RKE}}

\nc{\ShvCatAco}{\ShvCatA_{\co}}
\nc{\ShvCatHco}{\ShvCatH_{\co}}

\nc{\Stk}{\mathsf{Stk}}

\renc{\2}{(\infty,2)}
\renc{\1}{(\infty,1)}

\nc{\doubleCat}{\mathsf{doubleCat}}
\nc{\Spaces}{\on{Spc}}

\nc{\ALG}{\mathsf{ALG}}
\nc{\MAPS}{\mathsf{MAPS}}

\nc{\CAT}{\mathsf{CAT}}

\nc{\oneCat}{{\Cat_{\1}}}
\nc{\oneCAT}{{\CAT_{\1}}}

\nc{\twoCat}{{\Cat_{\2}}}
\nc{\twoCAT}{{\CAT_{\2}}}

\nc{\DGCAT}{\mathsf{DGCAT}}
\nc{\twoCatDG}{{\CAT_{\2}^\DG}}
\nc{\twoCATDG}{{\CAT_{\2}^\DG}}

\nc{\twoCATDGw}{{\CAT_{\2, w*}^\DG}}
\nc{\twoCATDGww}{{\CAT_{\2, ww*}^\DG}}

\nc{\AlgBimod}{\Alg^{\mathit{bimod}}}
\nc{\AlgBimodDGCat}{\AlgBimod(\DGCat)}
\nc{\ALGBimod}{\ALG^{\mathit{bimod}}}
\nc{\twoAlgBimod}{\ALGBimod}

\nc{\rev}{{\on{rev}}}

\nc{\lfp}{{\mathit{lfp}}}

\nc{\RBeck}{{\on{R-BC}}}
\nc{\LBeck}{{\on{L-BC}}}

\nc{\schem}{\mathit{schem}}
\nc{\proper}{\mathit{proper}}

\nc{\res}{{\mathit{res}}}

\nc{\UQCoh}{\U^{\QCoh}}
\nc{\UQ}{\UQCoh}

\nc{\LieAlgbd}{{\on{Lie-algbd}}}

\nc{\LY}{{L\Y}}

\nc{\TangQ}{\Tang^{\QCoh}}

\nc{\Fil}{{\on{Fil}}}
\nc{\AssGr}{\on{assoc-gr}}

\nc{\red}{{\mathit{red}}}

\nc{\Cech}{\on{Cech}}

\nc{\FormMod}{\mathsf{FormMod}}
\nc{\FormModunderStkevcoc} {\FormMod_{\Stk^{< \infty}/}^\lfp }

\nc{\vDmod}{\virg{\Dmod}}

\nc{\Betti}{{\on{Betti}}}

\nc{\LSG}{\LS_G}

\nc{\PP}{\mathbb{P}}

\nc{\Tr}{{\T\!\on r}}

\nc{\Bord}{\on{Bord}}
\nc{\Bordfr}{\on{Bord}^{\mathsf{fr}}}
\nc{\Bordor}{\on{Bord}^{\mathsf{or}}}
\nc{\hleftto}{\hookleftarrow}
\nc{\hlefto}{\hleftto}
\nc{\longleftto}{\longleftarrow}
\nc{\longlefto}{\longleftto}

\nc{\ICOHFORM}[2]
{
\ICoh_0
\Bigt{
\bigt{
#2
}
^\wedge_{#1}
}
}

\nc{\ICOHTOP}[2]
{
\ICoh_0
\Bigt{
\bigt{
Y^{#2}
}
^\wedge_{Y^{#1}}
}
}

\nc{\partialin}{\partial_{\mathit{in}}}
\nc{\partialout}{\partial_{\mathit{out}}}
\nc{\pin}{\partialin}
\nc{\pout}{\partialout}

\nc{\fE}{{f \! E}}

\nc{\Cob}{\on{Cob}}
\nc{\Cobor}{\on{Cob}^{\mathsf{or}}}
\nc{\RR}{\mathbb{R}}

\title[Chiral homology of the spherical category]{The topological chiral homology of the spherical category}
\author{Dario Beraldo}

\begin{document}
\maketitle

\begin{abstract}
We consider the spherical DG category $\Sph_G$ attached to an affine algebraic group $G$. By definition, $\Sph_G := \ICoh(\LS_G(S^2))$ consists of ind-coherent sheaves on the (derived) stack of $G$-local systems on the $2$-sphere $S^2$. The $3$-dimensional version of the pair of pants endows $\Sph_G$ with an $E_3$-monoidal structure.  
More generally, for an algebraic stack $Y$ and $n \geq -1$, we consider the $E_{n+1}$-monoidal DG category $\Sph(Y,n) := \ICoh_0((Y^{S^n})^\wedge_Y)$, where $\ICoh_0$ is the sheaf theory introduced by Arinkin and Gaitsgory.
The case of $\Sph_G$ is recovered by setting $Y =BG$ and $n=2$.

The cobordism hypothesis associates to $\Sph(Y,n)$ an $(n+1)$-dimensional topological field theory, whose value on a manifold $M^d$ of dimension $d \leq n+1$ (possibly with boundary) is given by the \emph{topological chiral homology} $\int_{M^d} \Sph(Y,n)$.
In this paper, we compute such chiral homology, obtaining the Stokes style formula
$$
\int_{M^d} \Sph(Y,n)
\simeq
\ICOHTOP{M^d}{\partial(M^d \times D^{n+1-d})},
$$
where the formal completion is constructed using the obvious projection $\partial(M^d \times D^{n+1-d}) \to M^d$. 
The most interesting instance of this formula is for $\Sph_G \simeq \Sph(BG,2)$, the original spherical category, and $X$ a Riemann surface. In this case, we obtain a monoidal equivalence $\int_X \Sph_G \simeq \H(\LSG^{\Betti}(X))$, where $\LSG^{\Betti}(X)$ is the stack of $G$-local systems on the topological space underlying $X$ and $\H$ is a sheaf theory related to Hochschild cochains.
\end{abstract}

\sec{Introduction}

\ssec{The main result}

Let $\kk$ be a field of characteristic zero. For $G$ an affine algebraic group over $\kk$, with Lie algebra $\g$, consider the \emph{spherical (or Satake) DG category} 
$$
\Sph_G
:=
\ICoh (BG \times_{\g/G} BG),
$$
presented in this form in \cite{AG}. To be precise on the terminology, $\Sph_G$ is called the \virg{renormalized spherical category} in \cite{AG}, and denoted there by $\Sph_G^\ren$. A slighly different form of $\Sph_G$ appeared earlier in \cite{BF}.

\medskip

The goal of this article is to \emph{integrate} this DG category over a Riemann surface. Let us explain what we mean by this.

\sssec{}

One checks that there is an equivalence $\Sph_G \simeq (\Sym(\g[-2]) \mod)^G$ of plain DG categories. However, $\Sph_G$ admits an $E_3$-monoidal structure which becomes evident in the realization
\begin{equation} \label{eqn:Sph_G written using S2}
\Sph_G 
\simeq
\ICoh
\bigt{
(BG^{S^2})^\wedge_{BG}
},
\end{equation}
where we have used the obvious isomorphisms
$$
BG \times_{\g/G} BG
\simeq
BG \times_{G/G} BG
\simeq
BG \times_{LBG} BG
\simeq
BG^{S^2}
\simeq
(BG^{S^2})^\wedge_{BG}.
$$
The $E_3$-structure comes from the $3$-dimensional version of the pair of pants contruction (or better from the space of configurations of little $3$-disks in a $3$-dimensional space), together with the functoriality of $\ICoh$ on formal completions (see \cite{book}).

\sssec{}

As explained in \cite{HA}, \cite{cobordism-hyp}, \cite{fact-hom}, any\footnote{
More precisely: $\A$ needs to be \emph{$SO(n+1)$-invariantly} $E_{n+1}$-monoidal. Otherwise, its topological chiral homology is only defined on $(n+1)$-framed manifolds. Luckily, $\Sph_G$ is $SO(3)$-invariantly $E_3$-monoidal, and likewise for all its generalizations $\Sph(Y,n)$ to be introduced later on.} 
$E_{n+1}$-monoidal DG category $\A$ can be \emph{integrated} on a closed oriented manifold $M$ of dimension $d \leq n+1$: the result will be an $E_{n+1-d}$-monoidal DG category, called the \emph{topological chiral homology of $\A$} and denoted by $\int_M \A$. 

In particular, it makes sense to compute the topological chiral homology of $\Sph_G$ on a Riemann surface $X$. The result will be a monoidal DG category, whose explicit calculation is the subject of our main result:

\begin{thm} \label{main-thm-intro}
For any affine algebraic group $G$ over $\kk$ and any Riemann surface $X$, there is a natural monoidal equivalence
\begin{equation} \label{main-formula-intro}
\int_{X}
\Sph_G
\simeq
\H(\LSG^{\Betti}(X)),
\end{equation}
where $\LSG^{\Betti}(X)$ is the derived stack of $G$-local systems on the topological space underlying $X$, and 
$$
\H(Y) := (\ICoh_0((Y \times Y)^\wedge_Y) ,\star)
$$
is the monoidal DG category introduced in \cite{centerH}.
\end{thm}

\begin{rem}
This theorem can be regarded as a topological instance of the main result of \cite{Nick-thesis}.
\end{rem}

\begin{rem} \label{rem: naive}
Instead of the spherical category $\Sph_G$, one might consider its simpler version $\QCoh(\pt/G)\simeq \Rep_G$, obtained by discarding the derived structure. Of course, $\Rep_G$ is symmetric monoidal. However, to keep the analogy with the above, let us just consider the underlying $E_3$-structure.  By \cite{BFN}, there is a monoidal equivalence:
\begin{equation} \label{naive formula}
\int_X \Rep_G
\simeq
\QCoh(\LSG^{\Betti}(X)).
\end{equation}
Note that $\Sph_G$ is a refinement of $\Rep_G$ (in that there is a monoidal map $\Rep_G \to \Sph_G$ that generates the target under colimits) and similarly $\H(Y)$ is a refinement of $\QCoh(Y)$ (again, there is a monoidal functor $\QCoh(Y) \to \H(Y)$ that generates the target under colimits, see \cite{centerH}). Thus, we may regard \eqref{main-formula-intro} as a refinement of \eqref{naive formula}. 
\end{rem}

Our interest in Theorem \ref{main-thm-intro} comes from its interplay with two different topics. The first topic, mentioned next, is the geometric Langlands program. The second one, discussed in Section \ref{sssec:higher Dmodules}, is the theory of shifted symplectic stacks and their quantizations (after \cite{PTVV}, \citep{CPTVV}).

\ssec{Langlands motivation}

\sssec{}

In this section only, we assume that $G$ is connected and reductive, with Langlands dual $\Gch$. 
Recall the rough statement of the Betti geometric Langlands conjecture (see \cite{Betti} for a thorough discussion): for any smooth complete curve $X$, there is an equivalence
$$
\LL^\Betti_G:
\ICoh_\N(\LS_{\Gch}^{\Betti}(X))
\xto{\; \; \simeq ? \; \; }
\Dmod^{\Betti}(\Bun_G(X)).
$$
\sssec{}

Now, as explained in \cite{centerH} and \cite{shvcatHH}, there is a tautological \virg{Hochschild} action of $\H(\LS_{\Gch}^{\Betti}(\Sigma))$ on $\ICoh_\N(\LS_{\Gch}^{\Betti}(\Sigma))$. Combining this with the above conjecture, we expect that $\H(\LS_{\Gch}^{\Betti}(\Sigma))$ acts on $\Dmod^{\Betti}(\Bun_G(\Sigma))$ as well, in a natural way (via the \virg{Hecke} action).

\medskip

To construct such action, we render the strategy outlined in \cite[Introduction]{shvcatHH} to the much easier Betti case. More precisely, we combine three ingredients: 
\begin{itemize}
\item
derived geometric Satake, see \cite{BF} and \cite{AG} for the version we will use;
\item
our main Theorem \ref{main-thm-intro};
\item
the main result of \cite{Nadler-Yun}, which, starting from \eqref{naive formula}, constructs an action of $\QCoh(\LSGch)$ on $\Dmod^{\Betti}(\Bun_G(\Sigma))$.
\end{itemize}
Details will be provided in the sequel to this paper.

\sssec{}

The action of $\H(\LS_{\Gch}^{\Betti}(X))$ on $\Dmod^{\Betti}(\Bun_G(X))$ brings several new tools to the Betti geometric Langlands program: the notion of singular support over $\LSGch$ for objects of $\Dmod^{\Betti}(\Bun_G(X))$, the notion of categorified Eisenstein series \cite{shvcatHH}, the \emph{strong spectral gluing theorem} \cite{gluing}, and so on. We hope to exploit these tools to give an explicit construction of the conjectural functor $\LL_G^\Betti$.

\ssec{Structure of the paper}

Section \ref{sec:background} is devoted to recalling some background notions.
In Section \ref{sec:TFT}, we extend the statement of Theorem \ref{main-thm-intro} to higher dimensional spheres and stacks other than $BG$. We explain how any of these generalized spherical categories yields a topological field theory (TFT).  Next, in Section \ref{sec:excision}, we prove our main technical result, Theorem \ref{thm:gluing along boundary}, which guarantes that $\ICoh_0$ satisfies the \emph{excision property}.
Finally, in Sections \ref{sec:spheres} and \ref{sec:Riemann-surfaces}, we exploit such property to compute the value of our TFT on various manifolds: spheres, tori, pairs of pants in various dimension, Riemann surfaces.

\ssec{Acknowledgements}
I am grateful to Nick Rozenblyum and Yakov Varshavsky for a stimulating conversation (Paris, January 2018) that triggered the writing of this paper.
I would also like to thank Clark Barwick, Andre Henriques, Claudia Scheimbauer and the anonymous referee for their generous help.
Research partially supported by EPSRC programme grant EP/M024830/1 Symmetries and Correspondences.

\sec{Notation and background} \label{sec:background}

Let us collect here some of the notation and background that will be used without further mention in the main body of the paper.

\ssec{Topology}

\sssec{}

We denote by $\Spaces$ the $\infty$-category of topological spaces as in \cite{HTT}. 
We say that a space $S \in \Spaces$ is \emph{finite} if it has the homotopy type of a finite simplical complex. We denote by $D^{k}$ and $S^k$ disks and spheres as usual.

\sssec{}
When we say that $S_1, S_2 \in \Spaces$ are isomorphic (written $S_1 \simeq S_2$), we actually mean that they are \virg{weak equivalent}.

\sssec{}

Manifolds are always assumed to be smooth. For a (smooth) manifold $M$, we denote by $\partial M$ its boundary. A manifold $M$ is said to be \emph{closed} if $\partial M = \emptyset$.
When gluing two manifolds along a common boundary, we always assume that the gluing has been performed in such a way that the result is also a manifold.

\sssec{}

We use the following notation: for an $n$-dimensional manifold $M$ possibly with boundary, we denote by $M^\circ$ the same manifold with a small $n$-dimensional disk removed from its interior. Similarly, the notation $M^{\circ\circ}$ means that we have removed two disjoint disks.
For instance, $(D^2)^{\circ \circ}$ is the usual pair of pants.

\ssec{Formal algebraic geometry}

We make essential use of the notion of formal completion in derived algebraic geometry and of the theory of ind-coherent sheaves.
For the former, we refer to \cite{book}, for the latter to \cite{ICoh} and \cite{book}.
Our notation follows those treatments, as well as \cite{centerH} and \cite{shvcatHH}.

\sssec{}

We always work over a field $\kk$ of characteristic zero.
All stacks in this paper are defined over $\kk$, and they are assumed to be derived, algebraic, quasi-compact, with affine diagonal and with perfect cotangent complex. 
We say that a stack is \virg{bounded} if it is \virg{eventually coconnective} in the terminology of \cite{book}. Recall that any quasi-smooth stack is bounded.

Even if we start with a smooth (in particular, underived) stack, the operations performed in this paper will quickly lead to unbounded (in particular, derived) stacks. For instance, the operation $Y \squigto LY$ of taking loops sends smooth stacks to quasi-smooth stacks and (genuinely) quasi-smooth stacks to unbounded stacks.

\sssec{}

We denote by $Z^\wedge_{Y}$ the formal completion of a map of stacks $f: Y \to Z$. Observe that $f$ is not required to be a closed embedding, and that the notation $Z^\wedge_Y$ is abusive as the map $f$ is not mentioned. When we suspect  that $f$ might not clear from the context, we write it down explicitly.

We assume familiarity with the theory of ind-coherents sheaves on schemes, stacks and formal completions thereof: see \cite{ICoh}, \cite{book}.

\sssec{}

A key player in this paper is the hybrid sheaf theory $\ICoh_0$ (a mixture between $\ICoh$ and $\QCoh$) on maps of stacks, introduced in \cite{AG2}. If $Y \to Z$ is a map of stacks with $Y$ bounded, then 
$$
\ICoh_0(Y \to Z)
=: \ICoh_0(Z^\wedge_Y)
$$
is defined as the full subcategory of $\ICoh(Z^\wedge_Y)$ given by the fiber product
$$
\ICoh_0(Z^\wedge_Y)
:=
\ICoh(Z^\wedge_Y) \ustimes{\ICoh(Y)} \QCoh(Y).
$$
On the other hand, if $Y$ is not bounded, such fiber product is not to be considered, and the definition of $\ICoh_0(Z^\wedge_Y)$ must be modified as described in \cite{centerH}.

\begin{example}

In particular, on any unbouded $Y$ there are two different DG categories of $\fD$-modules: the usual one $\Dmod(Y)$, see \cite{book}, and the more exotic one
$$
\vDmod(Y)
:=
\ICoh_0(\pt^\wedge_Y),
$$
introduced in \cite{centerH}. Any category of $\fD$-modules occurring in this paper must be interpreted as the exotic one.

\end{example}

\sssec{} \label{sssec:functoriality of ICOH0}

The assignment $\ICoh_0(Y \to Z)$ upgrades to a functor out of correspondences of arrows of stacks. Any diagram
$$ 
\begin{tikzpicture}[scale=1.5]
\node (00) at (0,0) {$Y_1$}; 
\node (10) at (1.5,0) {$Y_2$}; 
\node (20) at (3,0) {$Y_2$};
\node (01) at (0,.8)  {$Z_1$};
\node (11) at (1.5,.8) {$Z'$};
\node (21) at (3,.8) {$Z_2$};
\path[-> ,font=\scriptsize,>=angle 90]
(11.east) edge node[below] {$ $} (21.west);
\path[-> ,font=\scriptsize,>=angle 90]
(10.east) edge node[above] {$\simeq$} (20.west);
\path[ <-,font=\scriptsize,>=angle 90]
(01.east) edge node[below] {$ $} (11.west);
\path[ <-,font=\scriptsize,>=angle 90]
(00.east) edge node[below] {$ $} (10.west);
\path[<-,font=\scriptsize,>=angle 90]
(01.south) edge node[right] {} (00.north);
\path[<-,font=\scriptsize,>=angle 90]
(11.south) edge node[below] {$ $} (10.north);
\path[<-,font=\scriptsize,>=angle 90]
(21.south) edge node[below] {$ $} (20.north);
\end{tikzpicture}
$$
\emph{with cartesian left square} gives rise to a pull-push functor
$$
\ICoh_0(Y_1 \to Z_1)
\longto
\ICoh_0(Y_2 \to Z_2),
$$
see \cite{centerH}.

\begin{rem}
In the main body of the paper we will use the notation $\ICoh_0(Z^\wedge_Y)$ in place of the more precise $\ICoh_0(Y \to Z)$.
\end{rem}

\ssec{Mixing topology and algebraic geometry}

\sssec{}

The $\infty$-category of stacks is cotensored over finite spaces. This means that, for $Y$ a stack and $S$ a finite space, there exists a well defined $Y^S := \Maps(S, Y)$, which can be computed as an iterated fiber product of various copies of $Y$. For instance, 
$$
Y^{S^1} = LY = Y \times_{Y \times Y} Y, 
\hspace{.4cm}
Y^{S^2} = Y \times_{LY} Y.
$$ 

\sssec{}

It is already clear from the introduction that our DG categories of interest arise via a $3$-step procedure: start with a map of finite spaces, take $\Maps$ into a stack $Y$, take $\ICoh_0$ of the result. Diagrammatically, 
$$
[M \to N]
\squigto 
[Y^N \to Y^M]
\squigto
\ICOHTOP{N}{M}.
$$

\begin{rem}
It is often the case that $M$ and $N$ are manifolds. We emphasize that the resulting $\ICOHTOP{N}{M}$ depends only on the topological spaces underlying $M$ and $N$.
\end{rem}

\sssec{}

It follows from Section \ref{sssec:functoriality of ICOH0} that any commutative diagram of finite spaces
\begin{equation} \label{diag:cospans compatible}
\begin{tikzpicture}[scale=1.5]
\node (00) at (0,0) {$N_1$}; 
\node (10) at (2,0) {$N_2$}; 
\node (20) at (4,0) {$N_2$};
\node (01) at (0,1)  {$M_1$};
\node (11) at (2,1) {$M'$};
\node (21) at (4,1) {$M_2$};
\path[<-  ,font=\scriptsize,>=angle 90]
(11.east) edge node[below] {$ $} (21.west);
\path[<- ,font=\scriptsize,>=angle 90]
(10.east) edge node[above] {$\simeq$} (20.west);
\path[ ->,font=\scriptsize,>=angle 90]
(01.east) edge node[below] {$ $} (11.west);
\path[ ->,font=\scriptsize,>=angle 90]
(00.east) edge node[below] {$ $} (10.west);
\path[->,font=\scriptsize,>=angle 90]
(01.south) edge node[right] {} (00.north);
\path[->,font=\scriptsize,>=angle 90]
(11.south) edge node[below] {$ $} (10.north);
\path[->,font=\scriptsize,>=angle 90]
(21.south) edge node[below] {$ $} (20.north);
\end{tikzpicture}
\end{equation}
with \emph{cocartesian leftmost diagram} gives rise to a pull-push functor
$$
\ICOHTOP {N_1}{M_1}
\longto
\ICOHTOP {N_2}{M'}
\longto
\ICOHTOP {N_2}{M_2}.
$$

\sssec{}

Luckily, all diagrams (\ref{diag:cospans compatible}) of topological spaces arising in this paper are of this form: with cocartesian leftmost square and with bottom rightmost map an isomorphism (of spaces). We simply call these diagrams \emph{pairs of compatible cospans}.

\sec{The TFT perspective} \label{sec:TFT}

It turns out that the statement of Theorem \ref{main-thm-intro} can be generalized with no extra effort, using the language of topological field theories (TFTs). Such generalization is actually both significant and helpful.

It is significant, as it allow to reinterpret the results of this paper via the theory of shifted symplectic stacks and their quantizations.

It is helpful, as it clarifies the structures involved in the proof of the theorem.\footnote{
For instance, in the case $n=2$ several circles appear for various different reasons and it is easy to confuse their roles. On the other hand, for higher $n$, some of the circles are replaced by $(n-1)$-dimensional spheres and less confusion is likely to arise: see for instance Corollary \ref{cor:chiral-S1}.
}

\medskip

In this section, we introduce $\Sph(Y,n)$, our generalization of $\Sph_G$. We then argue that it determines a fully extended $(n+1)$-dimensional TFT (in the sense of \cite{cobordism-hyp}) and describe such TFT explicitly in \virg{Stokes' terms}. 

\begin{rem}
Ideally, TFTs would associate numbers to top dimensional manifolds. We stress that our TFT is not of this sort: its value on a manifold of top(=$n+1$) dimension is a DG category. In general, $\Sph(Y,n)$ is not dualizable enough to yield a TFT that associates numbers to manifolds of top dimension. See, however, Section \ref{sssec:extension to dimension n+2}.
\end{rem}

\ssec{Modules for the ring of higher differential operators} \label{ssec: diffops}

As noted in the abstract, the DG category $\Sph_G =\ICoh((BG^{S^2})^\wedge_{BG})$ is an instance of the following general construction.

\sssec{}

Let $Y$ be a quasi-compact derived algebraic stack locally of finite presentation and with affine diagonal, fixed throughout the paper. 
For $n \geq -1$,  we define the $E_{n+1}$-monoidal 
$$
\Sph(Y,n) :=\ICoh_0((Y^{S^n})^\wedge_Y),
$$
with multiplication induced by a higher dimensional version of the pair of pants.
For an extensive discussion of the $E_{n+1}$-monoidal structure, we recommend \cite{BT}.
Note the appearance of $\ICoh_0$, introduced in \cite{AG2}\footnote{and in \cite{centerH} in the case $Y$ is not bounded}, instead of simply $\ICoh$. These two sheaf theories concide for $Y$ smooth, but differ otherwise. 
It turns out that $\ICoh_0$ is much more amenable to gluing: see for instance \cite{centerH}, where the Drinfeld center of $\Sph(Y,0)$ is computed.

\begin{example}
For $n= 0$, we have $\Sph(Y,0) \simeq \H(Y)$ monoidally. For $n= -1$, we agree that $S^{-1} = \emptyset$ and that $Y^{\emptyset} = \pt$, whence $\Sph(Y,-1) \simeq \vDmod(Y)$.
\end{example}

\begin{example}
The choice $Y=BG$ and $n=2$ recovers $\Sph_G$. The reader will not lose much by considering $Y=BG$ thoughout this paper. However, the appearance of $\ICoh_0$ is \emph{unavoidable} even for $BG$, as soon as we integrate in $2$-dimensions: indeed, 
the stack of maps from a $2$-dimensional surface to $BG$ is quasi-smooth but not smooth (unless $G$ is trivial).
\end{example}

\sssec{} \label{sssec:higher Dmodules}

As we presently explain, $\Sph(Y,n)$ is a basic object of interest in shifted symplectic geometry (\cite{PTVV}, \cite{CPTVV}): a quantization of the $(n+1)$-shifted cotangent bundle of $Y$. For simplicity, let us assume that $Y$ is smooth, so that there is no difference between $\ICoh_0((Y^{S^n})^\wedge_Y)$ and $\ICoh((Y^{S^n})^\wedge_Y)$.
Then, by \cite[Chapter IV]{book}, 
$$
\Sph(Y,n)
=
\ICoh_0((Y^{S^n})^\wedge_Y)
\simeq
\U(\Tang_{Y/ (Y^{S^n})})
\mod 
(\ICoh(Y)),
$$
where $\Tang_{Y/ (Y^{S^n})}$ is the relative tangent Lie algebroid and $\U(\Tang_{Y/ (Y^{S^n})})$ its universal envelope. The latter is a monad acting on $\ICoh(Y)$, equipped with a canonical filtration. The appropriate version of the PBW theorem states that the associated graded equals the functor of tensoring with the symmetric algebra of $\Tang_{Y/ (Y^{S^n})} \in \ICoh(Y)$. Finally, one computes that $\Tang_{Y/ (Y^{S^n})} \simeq \Tang_Y[-(n+1)]$. Thanks to
$$
\Sym \Tang_Y[-(n+1)] \mod(\QCoh(Y)) \simeq \QCoh(T^*[n+1]Y),
$$
we obtain that $\Sph(Y,n)$ is indeed a quantization of the $(n+1)$-shifted cotangent bundle of $Y$. As such, $\Sph_G$ might be called the DG category of \emph{$(n+1)$-shifted $D$-modules} on $Y$.

\ssec{The extended spherical TFT}

\sssec{}

Recall the $(\infty, n+1)$-category $\Alg_{(n+1)}^\circ (\DGCat)$: this is the Morita $(\infty,n+1)$-category of $E_{n+1}$-monoidal DG categories (\cite[Definition 4.1.11]{cobordism-hyp},  \cite{Rune}, \cite{claudia-thesis}).

\sssec{}

Recall moreover that, by \cite[Theorem 4.1.24]{cobordism-hyp}, any $\A \in \Alg_{(n+1)}^\circ (\DGCat)$ gives rise to an $(n+1)$-dimensional TFT, that is to a symmetric monoidal functor
$$
\bT_{\A}:
\Bordfr_{n+1}
\longto
\Alg_{(n+1)}^\circ (\DGCat)
$$
defined on $(n+1)$-framed manifolds. Such functor is computed by topological chiral homology (alias: factorization homology). Besides \cite[Chapter 4.1]{cobordism-hyp}, see \cite{fact-hom}, \cite{AF}.

\sssec{}

The orthogonal group $O(n+1)$ acts on the space of $E_{n+1}$-monoidal DG categories, simply because it acts (in the obvious way) on the operad $E_{n+1}$. If an $E_{n+1}$-monoidal DG category $\A \in \Alg_{E_{n+1}}(\DGCat)$ admits a lift along the forgetful functor
$$
\bigt{
\Alg_{E_{n+1}}(\DGCat)
}^{SO(n+1)}
\longto
\Alg_{E_{n+1}}(\DGCat),
$$
then $\bT_\A$ descends to a TFT defined on oriented manifolds. 
In our case, we have
$$
\bigt{
\Alg_{E_{n+1}}(\DGCat)
}^{SO(n+1)}
\simeq
\Alg_{\fE_{n+1}}(\DGCat),
$$
where $\fE_{k} := SO(k) \ltimes E_k$ is the operad of framed little $k$-disks (\cite{CV}, \cite{Getzler}).

\sssec{}

It is tautological  that our $\Sph(Y,n)$ is actually an $\fE_{n+1}$-monoidal DG category,
whence we obtain a TFT
$$
\bT_n := \bT_{\Sph(Y,n)}: 
\Bordor_{n+1}
\longto 
\Alg^\circ_{(n+1)}(\DGCat),
\hspace{.4cm}
M 
\squigto \int_M \Sph(Y,n).
$$
For short, we sometimes call such theory the \emph{spherical TFT}. In particular, for $n=1$, we have the \emph{circular TFT} discussed in Section \ref{ssec:circular}.

\sssec{}

We claim that the entire theory $\bT_n$ can be described really explicitly in terms of $\ICoh_0$: this is the content of Claim  \ref{magic quasi-theorem} below. Before giving the answer, let us recall some of the structure we expect to find; in other words, let us recollect some of the features of the higher category $\Alg_{(n+1)}^\circ(\DGCat)$.
Let $0 \leq d \leq n+1$ and $\A$ a $\fE_{n+1}$-monoidal DG category. Then:

\begin{itemize}

\item

the topological chiral homology $\int_M \A$ on a closed oriented $d$-dimensional manifold $M$ is naturally an $E_{n+1-d}$-algebra;

\smallskip 

\item
for $M$ a $d$-dimensional oriented bordism between $B_1$ and $B_2$, the topological chiral homology $\int_M \A$ is an $E_{n+1-d}$-algebra in the $E_{n+1-d}$-monoidal $\infty$-category
$$
\Bigt{
\int_{B_1} \A, \int_{B_2} \A
} \bbimod.
$$
\end{itemize}

\begin{qthm} \label{magic quasi-theorem}
\footnote{In the present paper, we do not fully prove this claim, the (only) obstruction being the higher categorical nature of the statement. Instead, we prove several instances of it in the following sections.}
For $0 \leq d \leq n+1$ and $M$ an oriented $d$-dimensional manifold (possibly with boundary), we have
\begin{equation} \label{main-formula}
\int_{M^d} \Sph(Y,n)
\simeq
\ICoh_0
\bigt{
(Y^{\partial(M \times D^{n+1-d})   })^\wedge_{Y^M}
},
\end{equation}
and the two pieces of structure listed above are as follows.

\smallskip

If $M$ has no boundary, the above formula reduces to
\begin{equation} \label{eqn: no boundary}
\int_{M} \Sph(Y,n) 
\simeq 
\ICoh_0
\bigt{(Y^{M \times S^{n-d}})^\wedge_{Y^M}}
\simeq \Sph(Y^M, n-d).
\end{equation}
and the $E_{n-d+1}$-monoidal structure is the obvious one coming from boundaries of little disks in $D^{n+1-d}$.

\smallskip

If $M$ is viewed as a bordism between $B_1$ and $B_2$, the right action of 
$$
\int_{B_2} \Sph(Y,n)
 \simeq 
\ICoh_0
\bigt{(Y^{B_2 \times S^{n+1-d}})^\wedge_{Y^{B_2}}}
$$
on 
$$
\int_{M^d} \Sph(Y,n)
\simeq
\ICoh_0
\bigt{
(Y^{\partial(M \times D^{n+1-d})   })^\wedge_{Y^M}
}
$$
is induced by the pair of compatible cospans
$$ 
\begin{tiny}
\begin{tikzpicture}[scale=1.5]
\node (00) at (0,0) {$M \sqcup B $}; 
\node (10) at (4,0) {$M \sqcup_B B \simeq M $}; 
\node (20) at (7,0) {$M$,};
\node (01) at (0,1)  {$ \partial(M \times D^{n+1-d}) \sqcup (B \times S^{n+1-d}) $};
\node (11) at (4,1) {$\partial(M \times D^{n+1-d}) \sqcup_{B \times D^{n+1-d}} (B \times S^{n+1-d}) $};
\node (21) at (7,1) {$\partial(M \times D^{n+1-d})$};
\path[<-  ,font=\scriptsize,>=angle 90]
(11.east) edge node[below] {$ $} (21.west);
\path[<- ,font=\scriptsize,>=angle 90]
(10.east) edge node[below] {$ $} (20.west);
\path[  ->,font=\scriptsize,>=angle 90]
(01.east) edge node[below] {$ $} (11.west);
\path[  ->,font=\scriptsize,>=angle 90]
(00.east) edge node[below] {$ $} (10.west);
\path[->,font=\scriptsize,>=angle 90]
(01.south) edge node[right] {} (00.north);
\path[->,font=\scriptsize,>=angle 90]
(11.south) edge node[below] {$ $} (10.north);
\path[->,font=\scriptsize,>=angle 90]
(21.south) edge node[below] {$ $} (20.north);
\end{tikzpicture}
\end{tiny}
$$
where the right map forming 
$$
\partial(M \times D^{n+1-d}) \sqcup_{B \times D^{n+1-d}} (B \times S^{n+1-d})
$$
is defined by attaching the leftmost hemisphere of $S^{n+1-d}$ to $D^{n+1-d}$.
The left action of $B_1$ is defined in the same way.

\end{qthm}

\sssec{Stokes interpretation}

To guess the formula (\ref{main-formula}), one runs the physics proof of Stokes' theorem: given $M$, replace points of $M$ with small $n$-dimensional spheres or better by $n$-dimensional cubes. Then the adjacent faces of those cubes cancel out. For instance, this procedure applied to $S^1$ gives rise to $S^1 \times S^{n-1}$: this yields the statement of Corollary \ref{cor:chiral-S1}. The same procedure for $n=2$ and $X$ a Riemann surface yields $X \times S^0$, which gives in turn our main Theorem \ref{main-thm-intro}. For an example of a manifold with boundary, consider $n=2$ and the usual pair of pants $P$: then the procedure yields a Riemann surface of genus $2$, which is indeed the boundary of $P \times D^1$.

\sssec{AKSZ formalism} \label{sssec:AKSZ}

Another way\footnote{We are grateful to the referee for suggesting this point of view.} to explain \eqref{main-formula}, or rather \eqref{eqn: no boundary}, is by using the interpretation of $\Sph(Y,n)$ as the DG category of $(n+1)$-shifted $D$-modules. 
One of the main theorems of \cite{PTVV} states that the stack of maps from an $d$-oriented manifold to an $n$-shifted symplectic stack is $(n-d)$-shifted symplectic. This is an instance of the AKSZ formalism (\cite{AKSZ}). 
Applying this construction to the $(n+1)$-shifted cotangent space $T^*[n+1]Y$ and quantizing, we get the gist of \eqref{eqn: no boundary}: integrating over a closed oriented manifold $M^d$ sends $(n+1)$-shifted $D$-modules on $Y$ to $(n+1-d)$-shifted $D$-modules on the mapping stack $Y^M$.

\medskip

Below are two significant examples of \eqref{eqn: no boundary}.

\sssec{(Top dimension)} 

For $M^{n+1}$ a closed manifold of dimension $n+1$, the theorem yields an equivalence
$$
\int_{M^{n+1}} \Sph(Y,n) 
\simeq 
\vDmod(Y^{M})
$$
of DG categories.

\sssec{(Riemann surfaces)} \label{sssec:Riemann surfaces - statement}

If $n \geq 1$ and $X$ is a Riemann surface, then the TFT predicts an equivalence
$$
\int_{X} \Sph(Y,n) 
\simeq
\ICoh_0
\bigt{
(Y^{X \times S^{n-2}})^\wedge_{Y^X}
}
$$
of $E_{n-1}$-monoidal DG categories, where the $E_{n-1}$-structure on the RHS in induced by the pair of pants construction for $S^{n-2}$, together with the functoriality of $\ICoh_0$.

\medskip

In particular, for $n=2$, we get monoidal equivalence
$$
\int_{X} \ICoh((Y^{S^2})^\wedge_Y)
\simeq
\ICoh_0
\bigt{
(Y^{X \times S^0} )^\wedge_{Y^X}
}
\simeq
\ICoh_0
\bigt{
(Y^{X} \times Y^X )^\wedge_{Y^X}
},
$$
the RHS being monoidal under convolution. Since such monoidal DG category is $\H(Y^X)$ by definition, we see that Theorem \ref{main-thm-intro} is a particular case of Theorem \ref{magic quasi-theorem} applied to $Y=BG$.


\begin{rem}
As in Remark \ref{rem: naive}, it is instructive to compare \eqref{eqn: no boundary} with the formula for the topological chiral homology of $\QCoh(Y)$, the latter viewed as an $E_{n+1}$-monoidal DG category. Thanks to \cite{BFN}, we have $\int_{M^d} \QCoh(Y) \simeq \QCoh(Y^{M^d})$. Thus, \eqref{eqn: no boundary} refines this formula, in the same way as $D$-modules refine quasi-coherent sheaves.
\end{rem}

\ssec{Example: the \virg{circular} TFT} \label{ssec:circular}

Let us illustrate the example of $n=1$.

\sssec{}

In this case, attached to $\pt$ we have the \emph{circular category}, that is, the $E_2$-monoidal DG category 
$$
\Sph(Y,1)
:=
\ICOHFORM{Y}{LY}
\simeq
\Hcorr {LY} Y {\pt}.
$$
The monoidal structure is induced by the correspondence
$$
(LY)^\wedge_Y \times (LY)^\wedge_Y
\longleftto
(LY \times_Y LY)^\wedge_{Y}
\longto
(LY)^\wedge_Y,
$$
and the $E_2$-structure comes from the identification $Y^P \simeq LY \times_Y LY$, where $P$ is the pair of pants.

\sssec{}

We get a symmetric monoidal $(\infty,2)$-functor
$$
\bT_1:
\Bord_2
\longto
\Alg_{(2)}^\circ(\DGCat)
$$
described as follows:
\begin{itemize}
\item
the value on $\pt$ is the circular category;
\item
the value on $S^1$ is the $E_1$-monoidal DG category
$$
\bT_1(S^1) := 
\ICOHFORM{LY}{LY \times LY}
\simeq
\H(LY);
$$

\item

the value on an oriented $2$-dimensional surface $\Sigma$ with boundary $\partial S = \partialin S \sqcup \partialout S$ is the DG category
$$
\bT_1(S) := 
\ICOHFORM{Y^S}{Y^{\partialin S} \times  Y^{\partialout S} }
\simeq
\Hcorr{Y^{\partialin S}}{Y^S }{Y^{\partialout}},
$$
equipped with the natural $(\H(Y^{\partialin S}), \H(Y^{\partialout}))$-bimodule structure coming from the theory of $\H$ as in \cite{shvcatHH};

\item
in particular, the value on a closed oriented $2$-dimensional surface $S$ is the DG category
$$
\bT_1(S) := 
\ICOHTOP{S}{\emptyset}
\simeq
\ICOHFORM{Y^S}{\pt}
=:
\vDmod(Y^S).
$$

\end{itemize}

\sec{Excision for $\ICoh_0$} \label{sec:excision}

The most direct way to prove Claim \ref{magic quasi-theorem} would be to construct the TFT by hands, that is, to construct a symmetric monoidal $(\infty, n+1)$-functor from $\Bordor_{n+1} \to \Alg_{(n+1)}^\circ(\DGCat)$ that upgrades the assigment
\begin{equation} \label{eqn:assigment of spherical TFT}
M^d
\squigto
\ICOHTOP{M^d}{\partial(M^d \times D^{n+1-d})}.
\end{equation}
In the present paper, we content ourselves with a more modest and less structured approach, which however contains the main substance of the proof.\footnote{Indeed, it is conceivable that our excision theorem, combined with the results and techniques of \cite{AF}, actually gives Claim \ref{magic quasi-theorem}.} 
 Namely, we establish an excision (or gluing) theorem, Theorem \ref{thm:gluing along boundary}, that suffices for the proof Theorem \ref{main-thm-intro}.


\ssec{Gluing the $\ICoh_0$ categories}

Let us proceed to formulate the excision theorem. 

\sssec{}

Let $d \leq n$ and $M$ an oriented $(d+1)$-dimensional manifold, regarded as a bordism from $\pin M$ to $\pout M =: B$. 
As explained in Claim \ref{magic quasi-theorem}, the $E_{n-d+1}$-monoidal DG category 
$$
\ICOHTOP {B}{B \times S^{n-d}}
$$
acts on 
$$
\ICOHTOP {M}{ \partial(M \times D^{n-d})}
$$
from the right. 

\begin{rem}
As $n-d+1 \geq 1$, the former is in particular a monoidal DG category.
\end{rem}

\sssec{}

Now, let $N$ be another $(d+1)$-dimensional manifold, regarded as a bordism from $\pin N$ to $\pout N$. Assume we are given an identification
$$
B := \pout M \simeq \pin N.
$$ 
Then the $E_{n+1-d}$-monoidal DG category
$$
\ICOHTOP {B}{B \times S^{n-d}}
$$
acts on 
$$
\ICOHTOP {N}{ \partial(N \times D^{n-d})}
$$
from the left. The key point of this paper is to compute the relative tensor product
$$
\ICOHTOP {M}{ \partial(M \times D^{n-d})}
\usotimes{\ICOHTOP {B}{B \times S^{n-d}}}
\ICOHTOP {N}{ \partial(N \times D^{n-d})},
$$
which ought to correspond to the tensor product
$$
\int_M \Sph(Y,n)
\usotimes{\int_{B} \Sph(Y,n)}
\int_N \Sph(Y,n).
$$
By excision for topological chiral homology, see \cite{fact-hom}, we expect this to be equivalent to $\int_{M \sqcup_B N} \Sph(Y,n)$. The following theorem shows this is indeed the case.

\begin{thm} \label{thm:gluing along boundary}
Let $M$, $N$, $B$ and $d$ be as above.
The functor
$$
\ICOHTOP {M}{ \partial(M \times D^{n-d})}
\otimes
\ICOHTOP {N}{ \partial(N \times D^{n-d})}
$$
$$
\longto
\ICOHTOP {M \sqcup_B N}{ \partial((M \sqcup_B N) \times D^{n-d})}
$$
induced by the compatible pair of cospans
$$ 
\begin{tiny}
\begin{tikzpicture}[scale=1.5]
\node (00) at (0,0) {$M \sqcup N $}; 
\node (10) at (3,0) {$M \sqcup_B N $}; 
\node (20) at (6,0) {$M \sqcup_B N$};
\node (01) at (0,1)  {$ \partial(M \times D^{n-d}) \sqcup \partial(N \times D^{n-d}) $};
\node (11) at (3,1) {$\partial(M \times D^{n-d}) \sqcup_B \partial(N \times D^{n-d}) $};
\node (21) at (6,1) {$\partial((M \sqcup_B N) \times D^{n-d})$};
\path[<-  ,font=\scriptsize,>=angle 90]
(11.east) edge node[below] {$ $} (21.west);
\path[<-  ,font=\scriptsize,>=angle 90]
(10.east) edge node[below] {$ $} (20.west);
\path[ ->,font=\scriptsize,>=angle 90]
(01.east) edge node[below] {$ $} (11.west);
\path[  ->,font=\scriptsize,>=angle 90]
(00.east) edge node[below] {$ $} (10.west);
\path[->,font=\scriptsize,>=angle 90]
(01.south) edge node[right] {} (00.north);
\path[->,font=\scriptsize,>=angle 90]
(11.south) edge node[below] {$ $} (10.north);
\path[->,font=\scriptsize,>=angle 90]
(21.south) edge node[below] {$ $} (20.north);
\end{tikzpicture}
\end{tiny}
$$
descends to an equivalence
\begin{equation}
\nonumber
\ICOHTOP {M}{ \partial(M \times D^{n-d})}
\usotimes{\ICOHTOP {B}{B \times S^{n-d}}}
\ICOHTOP {N}{ \partial(N \times D^{n-d})}
\end{equation}
\begin{equation} \label{eqn:relative tensor for gluing}
\xto{\; \; \simeq \; \;}
\ICOHTOP {M \sqcup_B N}{ \partial((M \sqcup_B N) \times D^{n-d})}.
\end{equation}
\end{thm}

\begin{rem}
It is clear from the construction that the equivalence \eqref{eqn:relative tensor for gluing} is compatible with the left and right actions of
$$
\ICOHTOP {\pin M }{\pin M\times S^{n-d}}
\; \;  \mbox{ and } \; \; 
\ICOHTOP {\pout N }{\pout N \times S^{n-d}}
$$
on both sides.
Also, a diagram chase with little disks guarantees that \eqref{eqn:relative tensor for gluing} is $E_{n-d}$-monoidal.
\end{rem}

\sssec{}

The proof of Theorem \ref{thm:gluing along boundary} occupies the entire Section \ref{ssec:proof-Thm-gluing along boundary}.  Before getting there, let us observe that the difficulty in computing the relative tensor product (\ref{eqn:relative tensor for gluing}) is the following: the actions of the \virg{actor} on the \virg{actees} and defined by pull-push functors, rather than by (say) simply pull-backs along some map. This prevents us from invoking any kind of descent for $\ICoh_0$.
To circumvent this, we shall exploit the monoidal functor
$$
(\QCoh(Y^{B}), \otimes)
\longto
\Bigt{
\ICOHTOP {B}{B \times S^{n-d}}, \star
}.
$$

\sssec{}

Let us articulate the above idea in more detail. There are a priori various ways in which $\QCoh(Y^{B})$, viewed as a symmetric monoidal DG category, can act on $\ICOHTOP {B}{B \times S^{n-d}}$. The first, completely canonical, is via the monoidal functor
$$
(\QCoh(Y^{B}), \otimes)
\longto
\Bigt{
\ICOHTOP {B}{B \times S^{n-d}}, \star
}.
$$
The second one depends on the choice of a point of $S^{n-d}$ (equivalently, a connected component of $S^{n-d}$): such choice gives rise to a map 
$$
\bigt{ Y^{B \times S^{n-d}} }^\wedge_{Y^B}
\longto
Y^B
$$
and hence to an action by $(!,0)$-pullback. We leave it to the reader to check (by standard base-change nonsense) that these two actions are identified.

\ssec{Proof of Theorem \ref{thm:gluing along boundary}} \label{ssec:proof-Thm-gluing along boundary}

We will proceed in several steps.

\sssec{}

A routine diagram chase with cospans shows that the functor in question descends to the relative tensor product as stated. Thus, we just need to verify that the functor appearing in (\ref{eqn:relative tensor for gluing}) is an equivalence of DG categories.

\sssec{}

Consider now the monoidal functor
$$
(\QCoh(Y^{B}), \otimes)
\longto
\Bigt{
\ICOHTOP {B}{B \times S^{n-d}}, \star
}.
$$
We claim that the resulting right action of $\QCoh(Y^{B})$ on $\ICOHTOP {M}{ \partial(M \times D^{n-d})}$ comes from the monoidal functor
$$
\QCoh(Y^{B})
\xto{r^*}
\QCoh(Y^M)
\longto
\ICOHTOP {M}{ \partial(M \times D^{n-d})},
$$
where $r: Y^M \to Y^{B}$ is the map induced by the inclusion $B \hto M$.
This is made evident by the pushout square
$$ 
\begin{tikzpicture}[scale=1.5]
\node (00) at (0,0) {$\partial(M \times D^{n-d})  \sqcup B $}; 
\node (10) at (4,0) {$\partial(M \times D^{n-d})  \sqcup_B B \simeq \partial(M \times D^{n-d}),$}; 
%
\node (01) at (0,1)  {$ \partial(M \times D^{n-d}) \sqcup (B \times S^{n-d}) $};
\node (11) at (4,1) {$\partial(M \times D^{n-d}) \sqcup_{B \times D^{n-d}} (B \times S^{n-d}) $};
%
\path[ ->,font=\scriptsize,>=angle 90]
(01.east) edge node[below] {$ $} (11.west);
\path[ ->,font=\scriptsize,>=angle 90]
(00.east) edge node[below] {$ $} (10.west);
\path[->,font=\scriptsize,>=angle 90]
(01.south) edge node[right] {} (00.north);
\path[->,font=\scriptsize,>=angle 90]
(11.south) edge node[below] {$ $} (10.north);

\end{tikzpicture}
$$
together with the fact that the map $B \to \partial(M \times D^{n-d})$ appearing in the bottom horizontal arrow factors as
$$
B \hto M \to \partial(M \times D^{n-d}).
$$
Parallel considerations hold for $M$ in place of $N$.

\sssec{}

To perform the computation without too much notation, let us set
\renc{\Q}{\mathcal Q}
\begin{eqnarray}
\nonumber
 & \, &
\B := \ICOHTOP {B}{B \times S^{n-d}},
\\
\nonumber
 & \, &
\M := 
\ICOHTOP {M}{ \partial(M \times D^{n-d})},
\\
\nonumber
 & \, &
\N :=
\ICOHTOP {N}{ \partial(N \times D^{n-d})},
\end{eqnarray}
as well as
\begin{eqnarray}
\nonumber
 & \, &
\A := \QCoh(Y^B),
\\
\nonumber
 & \, &
\Q := \QCoh(Y^M), 
\\
\nonumber
 & \, &
\R := \QCoh(Y^N). 
\end{eqnarray}
We need to compute $\M \otimes_\B \N$, bearing in mind that $\A$ maps monoidally to $\B$.
Hence, we can compute $\M \otimes_\B \N$ as a relative tensor product in the symmetric monoidal $\infty$-category $\A \mmod$:
\begin{equation} \label{eqn:tensor product as a bar construction over A}
\M \otimes_\B \N
\simeq
\uscolim{[m] \in \bDelta^\op}
\Bigt{
\M \otimes_{\A}
\underset{m \, \text{occurrences of $\B$}}{\underbrace{\B \Otimes_\A \cdots \Otimes_\A \B}}
\otimes_{\A} \N
}.
\end{equation}

\sssec{}

Recall now that the action of $\A$ on $\M$ (respectively, $\N$) comes from the monoidal functor $\A \to \Q \to \M$ (respectively, $\A \to \R \to \N$).
Then the zero simplex above is given by
\begin{eqnarray}
\nonumber
\M \otimes_{\A} \N
& \simeq &
\M \otimes_\Q \Q \otimes_{\A} \R \otimes_\R \N
\\
\nonumber
& \simeq &
\M \otimes_\Q \QCoh(Y^{M \sqcup_B N}) \otimes_\R \N.
\end{eqnarray}
Reverting to the $\ICoh_0$ notation, we have
\begin{eqnarray}
\nonumber
\M \otimes_{\A} \N
& \simeq &
\ICOHTOP {M}{ \partial(M \times D^{n-d})}
\usotimes{\QCoh(Y^M)}
\QCoh(Y^{M \sqcup_B N})
\usotimes{\QCoh(Y^N)}
\ICOHTOP {N}{ \partial(N \times D^{n-d})}
\\
\nonumber
& \simeq &
\ICOHTOP {M \sqcup_B N}{ \partial(M \times D^{n-d}) \sqcup_M M \sqcup_B N \sqcup_N \partial(N \times D^{n-d})}
\\
\nonumber
& \simeq &
\ICOHTOP {M \sqcup_B N}{ \partial(M \times D^{n-d})  \sqcup_B \partial(N \times D^{n-d})},
\end{eqnarray}
where we have used the theory of $\ICoh_0$ on possibly unbounded stacks as developed in \cite[Section 3]{centerH}.

\medskip

In passing, let us note that it is good news that the space appearing in the top part of the formal completion is exactly the space appearing at the center of the top cospan in the statement of Theorem \ref{thm:gluing along boundary}.

\sssec{}

The other simplices of the simplicial category appearing in (\ref{eqn:tensor product as a bar construction over A}) are no more difficult to compute. First off, for $m \geq 0$, we have
$$
\underset{m \, \text{occurrences of $\B$}}{\underbrace{\B \Otimes_\A \cdots \Otimes_\A \B}}
\simeq
\ICOHTOP{B}{ ((B \times S^{n-d})^{\sqcup_B m}) },
$$
where we have used the notation
$$
S^{\sqcup_B m}
:=
\underset{m \, \text{occurrences of $S$}}
{\underbrace{S \sqcup_B \cdots \sqcup_B S}},
\hspace{.4cm}
S^{\sqcup_B 0}
\simeq B.
$$
From this, we quickly deduce the equivalence
$$
\M \otimes_{\A}
\underset{m \, \text{occurrences of $\B$}}{\underbrace{\B \Otimes_\A \cdots \Otimes_\A \B}}
\otimes_{\A} \N
\simeq
\ICOHTOP 
{M \sqcup_B N}
{\partial(M \times D^{n-d}) 
 \sqcup_B
\bigt{ (B \times S^{n-d})^{\sqcup_B m} }
\sqcup_B
\partial(N \times D^{n-d})}.
$$

\sssec{}

Unraveling the equivalences, we see that $\M \otimes_\B \N$ is obtained from the tautological cosimplicial space (over $B$)
\begin{equation} \label{eqn:Cech-cosimplicial-diagram-of-spaces}
\partial(M \times D^{n-d}) 
 \sqcup_B
 \bigt{
(B \times S^{n-d})^{\sqcup_B m}
}
\sqcup_B
\partial(N \times D^{n-d})
\end{equation}
by the following procedure:
\begin{itemize}
\item
apply $\Maps(-,Y)$ to obtain a simplical stack over $Y^B$;
\item
formally complete at each stage (or better, regard the above as a simplcial object in $\Arr(\Stk)$);
\item
take the covariant $\ICoh_0$ to obtain a simplicial DG category;
\item
take the colimit (geometric realization).
\end{itemize}

\sssec{}

Consider now $\partial((M \sqcup_B N) \times D^{n-d})$ and the map
$$
\partial((M \sqcup_B N) \times D^{n-d})
\longto
\partial((M \sqcup_B N) \times D^{n-d})
\underset{B \times S^{n-d-1} } \sqcup
B
\simeq
\partial(M \times D^{n-d})  \sqcup_B \partial(N \times D^{n-d})
$$ 
defined by collapsing the central 
$B \times S^{n-d-1} \hto 
\partial((M \sqcup_B N) \times D^{n-d})
$ onto $B$.
It is clear that the Cech coresolution of such map identifies with the cosimplicial space featuring in (\ref{eqn:Cech-cosimplicial-diagram-of-spaces}).

\sssec{}

Since $\ICoh_0$ satisfies descent in the second variable (see \cite{centerH}), we conclude that
$$
\M \otimes_\B \N
\simeq
\ICOHTOP {M \sqcup_B N}{ \partial((M \sqcup_B N) \times D^{n-d})}
$$
as desired.

\sec{Spheres and tori} \label{sec:spheres}

With no doubt, Theorem \ref{thm:gluing along boundary} is the most important result of this paper. We now exploit it to compute the topological chiral homology of $\Sph(Y,n)$ on all spheres and tori of dimension $\leq n+1$. We start with the circle, in which case the calculation can be performed in greater generality. 
The case of higher dimensional tori follows by iterating this computation, while the case of spheres follows inductively using the hemisphere decomposition and Theorem \ref{thm:gluing along boundary}. 

Having determined the value of our theory $\bT_n$ on all spheres, we can show that $\bT_n$ \emph{further extends} to an $(n+2)$-dimensional theory that assigns vector spaces to closed oriented $(n+2)$-dimensional manifolds.

\ssec{Computing the trace}

For $A$ an $E_1$-monoidal DG category, we have $\int_{S^0} \A
\simeq
\A \otimes \A^\rev$ tautologically, and
$$
\int_{S^1} \A
\simeq
\A
\usotimes{\A \otimes \A^\rev} \A
=: \Tr(\A)
$$
by excision, see \cite{fact-hom}. Note that $\Tr(\A)$ is the Hochschild homology category of $\A$, also called the \emph{trace} in \cite{BFN}.

In our case, describing $\Tr(\Sph(Y,n))$ in tangible terms is easy thanks to Theorem \ref{thm:gluing along boundary}, or alternatively directly using the theory developed in \cite{centerH}. In fact, let us start by performing a more general calculation, which is a mild generalization of the main result of \cite{centerH}.

\begin{thm} \label{thm:trace of H relative Y/Z}
Consider the the convolution monoidal DG category
$$
\H(Y/Z) := \ICoh_0((Y \times_Z Y)^\wedge_Y),
$$
where $Y \to Z$ is a map of algebraic stacks with perfect cotangent complex.
Then
$$
\int_{S^1}
\H(Y/Z)
\simeq
\ICoh_0((LZ)^\wedge_{LY}),
$$
as DG categories.
\end{thm}

\begin{proof}
We exploit the monoidal functor $\QCoh(Y) \to \H(Y/Z)$ and compute the trace via the bar construction in the symmetric monoidal $\infty$-category $\QCoh(Y) \mmod$.
Namely:
$$
\Tr(\H(Y/Z))
\simeq
\uscolim{[n] \in \bDelta^\op}
\Bigt{
\H(Y/Z)^{(\otimes_{\QCoh(Y)})^{n+1}}
\usotimes{\QCoh(Y) \otimes \QCoh(Y)} 
\QCoh(Y)
}.
$$
Applying \cite[Proposition 3.4.5]{centerH} repeatedly, we obtain 
$$
\Tr(\H(Y/Z))
\simeq
\uscolim{[n] \in \bDelta^\op} \; 
\ICoh_0
\left(
\Bigt{
\underset{n+1 \, \text{times}}{\underbrace{Y \times_Z \cdots \times_Z Y}}
\times_Z LZ
}^\wedge_{LY}
\right).
$$
It is clear that such sequence of formal completions is the Cech resolution of the map $(Y \times_Z LZ)^\wedge_{LY} \to (LZ)^\wedge_{LY}$. Moreover, the structure functors are $(*,0)$-pushforwards along the structure maps of the Cech resolution.
By descent of $\ICoh_0$, we conclude that $\Tr(\H(Y/Z)) \simeq \ICoh_0((LZ)^\wedge_{LY})$ as wanted.
\end{proof}

\sssec{}

In particular, setting $Z = Y^{S^{n-1}}$ for $n \geq 0$, we see that $\H(Y/Z) \simeq \Sph(Y,n)$ and that $LZ = Y^{S^1 \times S^{n-1}}$. Thus, we obtain the following instance of Claim \ref{magic quasi-theorem}.
\begin{cor} \label{cor:chiral-S1}
For $n \geq 0$ and $Y$ an algebraic stack locally of finite presentation, there is a natural equivalence 
\begin{equation} \label{eqn:int of Sph over S1}
\int_{S^1} \Sph(Y, n)
\simeq
\ICoh_0
\Bigt{
\bigt{ Y^{S^1 \times S^{n-1}}
}
^\wedge_{Y^{S^1}}
}
\end{equation}
of $E_n$-monoidal DG categories, where the $E_n$-structure on the RHS is induced by the pair of pants construction for $S^{n-1}$, together with the functoriality of $\ICoh_0$.
\end{cor}

\begin{proof}
Only the $E_n$-structure needs an explanation. To this end, let us write down the candidate $E_n$-monoidal structure on
$$
\Tr(\Sph(Y,n))
\simeq
\ICoh_0
\Bigt{
\bigt{ Y^{S^1 \times S^{n-1}}
}
^\wedge_{Y^{S^1}}
}.
$$
Letting $P_n := (D^n)^{\circ \circ}$ denote the $n$-dimensional pair of pants, we consider the compatible pair of cospans:
$$ 
\begin{tikzpicture}[scale=1.5]
\node (00) at (0,0) {$S^1 \sqcup {S^1}$}; 
\node (10) at (3,0) {$S^1$}; 
\node (20) at (6,0) {$ S^1$.};
\node (01) at (0,1)  {$(S^1 \times S^{n-1}) \sqcup (S^1 \times S^{n-1}) $};
\node (11) at (3,1) {$S^1 \times P_n$};
\node (21) at (6,1) {$ S^1 \times S^{n-1}$};
\path[<-,font=\scriptsize,>=angle 90]
(11.east) edge node[below] {$ $} (21.west);
\path[<-,font=\scriptsize,>=angle 90]
(10.east) edge node[below] {$ $} (20.west);
\path[->,font=\scriptsize,>=angle 90]
(01.east) edge node[below] {$ $} (11.west);
\path[->,font=\scriptsize,>=angle 90]
(00.east) edge node[below] {$ $} (10.west);
\path[->,font=\scriptsize,>=angle 90]
(01.south) edge node[right] {} (00.north);
\path[->,font=\scriptsize,>=angle 90]
(11.south) edge node[below] {$ $} (10.north);
\path[->,font=\scriptsize,>=angle 90]
(21.south) edge node[below] {$ $} (20.north);
\end{tikzpicture}
$$
Then the $E_n$-monoidal structure on the trace is given by applying $\Maps(-, Y)$ to this diagram, and then by taking the $\ICoh_0$ pull-push functor.

\medskip

Let us also write the functor
\begin{equation} \label{eqn:functor from  A otimes A to Tr(A)}
\ICOHTOP {}{S^n}
\otimes
\ICOHTOP {}{S^n}
\longto
\ICOHTOP {S^1}{S^1 \times S^{n-1}}
\end{equation}
inducing the equivalence (\ref{eqn:int of Sph over S1}). This is given by the compatible map of cospans
$$ 
\begin{tikzpicture}[scale=1.5]
\node (00) at (0,0) {$\pt \sqcup \pt$}; 
\node (10) at (3,0) {$(S^1 \times D^{n}) \simeq S^1$}; 
\node (20) at (6,0) {$ S^1$.};
\node (01) at (0,1)  {$S^n \sqcup S^n$};
\node (11) at (3,1) {$(S^1 \times D^{n})^{\circ \circ}$};
\node (21) at (6,1) {$ S^1 \times S^{n-1}$};
\path[<-,font=\scriptsize,>=angle 90]
(11.east) edge node[below] {$ $} (21.west);
\path[<-,font=\scriptsize,>=angle 90]
(10.east) edge node[below] {$ $} (20.west);
\path[->,font=\scriptsize,>=angle 90]
(01.east) edge node[below] {$ $} (11.west);
\path[->,font=\scriptsize,>=angle 90]
(00.east) edge node[below] {$ $} (10.west);
\path[->,font=\scriptsize,>=angle 90]
(01.south) edge node[right] {} (00.north);
\path[->,font=\scriptsize,>=angle 90]
(11.south) edge node[below] {$ $} (10.north);
\path[->,font=\scriptsize,>=angle 90]
(21.south) edge node[below] {$ $} (20.north);
\end{tikzpicture}
$$

\medskip

Now, writing $\RR^{n+1} \simeq \RR \times \RR^n$, it suffices to notice that the action of $\Sph(Y,n) \otimes \Sph(Y,n)^\rev$ on $\Sph(Y,n)$ and the functor (\ref{eqn:functor from  A otimes A to Tr(A)}) \virg{happen} on the first factor, whence their are compatible with the complementary $E_n$-monoidal structure.
\end{proof}

\begin{example}
For $n=0$, we have $S^{-1} = \emptyset$ and consequently $Y^{\emptyset} = \pt$. Hence, in this case the theorem states that $\Tr(\H(Y)) := \int_{S^1}\H(Y) \simeq \vDmod(LY)$. This result had already been established in \cite{centerH}.
\end{example}

\sssec{The value on higher dimensional tori}

Iterating Theorem \ref{thm:trace of H relative Y/Z}, one immediately obtains the value of $\bT_n$ on all tori $T^d \simeq (S^1)^d$ of dimension $d \leq n+1$: an equivalence
$$
\int_{T^d} \Sph(Y,n)
\simeq
\ICOHTOP {T^d} {T^d \times S^{n-d}}
$$
of $E_{n-d+1}$-algebras.

\ssec{The higher Hochschild homologies of the spherical category}

We can now compute the value of $\bT_n$ on all spheres of dimension $ 1 \leq d \leq n+1$. For a general $E_n$-algebra $\A$, the invariant $\int_{S^d} \A$ is called the $E_d$-Hochschild homology of $\A$, see \cite{fact-hom}.

\sssec{}

In our case, we wish show that 
$$
\int_{S^d} \Sph(Y,n)
\simeq
\ICOHTOP {S^d}{S^d \times S^{n-d}}
$$
as $E_{n+1-d}$-monoidal DG categories.
Since the LHS enjoys excision in the form of
$$
\int_{S^d} \Sph(Y,n)
\simeq
\Sph(Y,n)
\usotimes{\int_{S^{d-1}} \Sph(Y,n)}
\Sph(Y,n),
$$
we need to show that so does the RHS. This is guaranteed by the following instance of 
Theorem \ref{thm:gluing along boundary}, together with induction on $d$ (the case $d = 0$ being trivial, the case $d=1$ being Corollary \ref{cor:chiral-S1}).

\begin{thm} \label{thm:integral over Sigma zero}
For $Y$ an algebraic stack locally of finite presentation and $0 \leq d \leq n+1$,  the $\ICoh_0$ pull-push along the compatible pair of cospans
\begin{equation} \label{diag:bordism-for-Sigma zero equivalence}
\begin{tikzpicture}[scale=1.5]
\node (00) at (0,0) {$\pt \sqcup \pt$}; 
\node (10) at (3,0) {$S^d \times
D^{n+1-d}\simeq S^d$}; 
\node (20) at (6,0) {$S^d$};
\node (01) at (0,1)  {$S^n\sqcup S^n$};
\node (11) at (3,1) {$(S^d \times D^{n+1-d})^{\circ \circ}$};
\node (21) at (6,1) {$S^d \times S^{n-d}$};
\path[<- left hook ,font=\scriptsize,>=angle 90]
(11.east) edge node[below] {$ $} (21.west);
\path[<- ,font=\scriptsize,>=angle 90]
(10.east) edge node[below] {$ $} (20.west);
\path[right hook ->,font=\scriptsize,>=angle 90]
(01.east) edge node[below] {$ $} (11.west);
\path[ ->,font=\scriptsize,>=angle 90]
(00.east) edge node[below] {$ $} (10.west);
\path[->,font=\scriptsize,>=angle 90]
(01.south) edge node[right] {} (00.north);
\path[->,font=\scriptsize,>=angle 90]
(11.south) edge node[below] {$ $} (10.north);
\path[->,font=\scriptsize,>=angle 90]
(21.south) edge node[below] {$ $} (20.north);
\end{tikzpicture}
\end{equation}
descends to an $E_{n-1}$-monoidal equivalence
$$
\nonumber
\ICOHTOP {}{S^n}
\usotimes
{\ICOHTOP {S^{d-1}}{S^{d-1} \times S^{n-d+1}}}
\ICOHTOP {}{S^n}
$$
\begin{equation} \label{eqn:tensor-prod-computing- int over Sigma zero}
\xto{\; \; \simeq \; \; }
\ICOHTOP {S^d}{S^d \times S^{n-d}}.
\end{equation}
\end{thm}

\begin{rem}
At this point, we have already proven Theorem \ref{main-thm-intro} in the case of genus zero: it suffices to set $Y=BG$ and $n=2$. Note moreover that $\LSG^\Betti(S^2) \simeq \LS_G(\PP^1)$ as algebraic stacks, where the latter is the stack of de Rham local systems on the algebraic curve $\PP^1$.
\end{rem}

\sssec{} \label{sssec:extension to dimension n+2}

From this description of $\int_{S^d} \Sph(Y,n)$, we see that the spherical TFT can be extended further to an $(n+2)$-dimensional TFT, that assigns a vector space to any closed oriented $(n+2)$-dinensional manifold. To see this, we invoke \cite[Remark 4.1.27]{cobordism-hyp}: we need to show that $\Sph(Y,n)$ is dualizable as a module over each of its Hochschild homologies. In our case this is clear, as $\Sph(Y,n)$ is dualizable as a DG category and each $\int_{S^d} \Sph(Y,n)$ is rigid.

\medskip

For instance, $\Sph(BG,0)$ gives rise to a $2$-dimensional theory that has been described in \cite{BZGN}.

\sec{Pairs of pants and Riemann surfaces} \label{sec:Riemann-surfaces}

In this last section, we wish to prove Theorem \ref{main-thm-intro}, or rather its generalization given in Section \ref{sssec:Riemann surfaces - statement}.
For these statements to make sense, we need $n \geq 1$, an assumption we keep in place henceforth.

\medskip

In other words, we wish to prove Claim \ref{magic quasi-theorem} on manifolds of dimension $\leq 2$. To this end, it is essential to come to grips with the value of the theory on the pair of pants. 

\medskip

Having control of $\int_{P_2} \Sph(Y,n)$, we apply Theorem \ref{thm:gluing along boundary} repeatedly to obtain the integral of $\Sph(Y,n)$ on any Riemann surface.

\ssec{Pairs of pants}

Thanks to a mild generalization of Theorem \ref{thm:gluing along boundary}, we can actually compute $\int_{P_d} \Sph(Y,n)$ for any $d$-dimensional pair of pants $P_d = (D^d)^{\circ \circ}$, with $1 \leq d \leq n+1$.\footnote{We expect this computation to be enough to prove Claim \ref{magic quasi-theorem} using a handle decomposition of $M$.} (The case $d=1$ is not interesting, so let us assume $d \geq 2$.)

\sssec{}

Let $\A$ an $E_{d-1}$-algebra. Then the topological chiral homology $\int_{P^d} \A$ comes with the structure of an $E_{d-1}$-algebra in the $E_{d-1}$-monoidal $\infty$-category 
$$
\left(
\int_{S^{d-1}} \A
\otimes
\int_{S^{d-1}} \A
,
\int_{S^{d-1}} \A
\right) \bbimod.
$$
Let us write down these three actions.

\sssec{}

Consider the decomposition $P^d \simeq (H^d)^\circ \sqcup_{D^{d-1}} (H^d)^\circ$, where $H^{d}$ denotes the $d$-dimensional hemisphere obtained by cutting $D^d$ is half.
Accordingly, the TFT guarantees that 
$$
\int_{P_d} \A
\simeq
\int_{ (H^d)^\circ} \A
\usotimes{\int_{(D^{d-1})} \A}
\int_{(H^d)^\circ} \A
\simeq
\int_{S^{d-1}} \A
\usotimes{\A}
\int_{S^{d-1}} \A,
$$
where the two actions of $\A \simeq \int_{D^{d-1}} \A$ on $\int_{S^{d-1}} \A \simeq \int_{ (H^d)^\circ} \A$ are determined by  modifications along the flat face of each of the holey hemispheres $(H^d)^\circ$.

\medskip

Now, imagining $P_d$ as a bordism from $S^{d-1} \sqcup S^{d-1}$ to $S^{d-1}$, the two \virg{incoming} actions of $\int_{S^{d-1}} \A$ on 
$$
\int_{S^{d-1}} \A
\usotimes{\A}
\int_{S^{d-1}} \A
$$
are the obvious ones induced by the action of $\int_{S^{d-1}} \A$ on itself. 

\sssec{}

The third action, the one coming from the \virg{outgoing} piece of the boundary, is slightly more complicated to write. Since the decompositon $
P^d \simeq (H^d)^\circ \sqcup_{D^{d-1}} (H^d)^\circ$ broke the symmetry, we also need to break $\int_{S^{d-1}} \A$ as
$$
\A
\usotimes{
\int_{S^{d-2}} \A
}
\A.
$$
Then the action of the latter on 
$$
\int_{S^{d-1}} \A
\usotimes{\A}
\int_{S^{d-1}} \A
$$
is given by the two componentwise actions of $\A$ on $\int_{S^{d-1}} \A$.

\sssec{}

Let us now apply this to $\A = \Sph(Y,n)$. Going through the proof of Theorem \ref{thm:gluing along boundary}, it is easy to realize that the statement remains valid for $M$ and $N$ manifolds with corners, \emph{as long as} the chose piece of boundary $B$ has no corners. This is the case in our situation: $P^d \simeq (H^d)^\circ \sqcup_{D^{d-1}} (H^d)^\circ$.

\medskip

From this we get the proof of Claim \ref{magic quasi-theorem} in the present case, that is, an equivalence
\begin{equation} \label{eqn:top-chiral-for-pants}
\int_{P_d}
\Sph(Y,n)
\simeq
\ICoh_0
\Bigt{
\bigt{
Y^{\partial(P \times D^{n+1-d})}
}
^\wedge_{Y^P}
}.
\end{equation}
It takes some unraveling to check that the three actions of $\int_{S^d} \Sph(Y,n)$ on $\int_{P^d} \Sph(Y,n)$ are given as in Claim \ref{magic quasi-theorem} applied to each component of the boundary of $P_d$.

\begin{example}
Let us look at the case of $n=2$ (our original case) and $d=2$, so that $P$ is the usual pair of pants. Let us consider $P_g := P^2_g$ to be a $g$-ary pair of pants: that is a disk with $g$ small disks removed.
 Then (a straightforward generalization of the) equivalence (\ref{eqn:top-chiral-for-pants}) takes the form of
$$
\int_{P_g}
\Sph(Y,2)
\simeq
\ICOHTOP
{(S^1)^{\vee g}}{X_g}
$$
where $X_g$ is a Riemann surface of genus $g$ and the bouquet of $g$ circles wraps around it in the obvious way. The $g+1$ actions of 
$$
\int_{S^1}
\Sph(Y,2)
\simeq
\ICOHTOP{S^1}{X_1}
$$
on $\int_{P_g}
\Sph(Y,2)$ are given by bubbling off $g+1$ tori: the \virg{incoming} $g$ ones are placed around the $g$ \virg{holes} of $X_g$,  the \virg{outgoing} one is placed around the entire $X_g$. 
\end{example}

\ssec{Topological chiral homology on surfaces}

Having established the value of the TFT attached to $\Sph(Y,n)$ on the pair of pants, we can now proceed to computing the theory in dimension $2$. We follow the argument of \cite[Section 3.2]{BZGN}.

\sssec{}

Let $\Bord_{1,2}$ denote the $\infty$-category of oriented bordisms of (oriented) surfaces. I.e., objects are unions of circles and morphisms are oriented surfaces with boundary given by circles (some ingoing, the rest outgoing).

We will look at two symmetric monoidal functors from $\Bord_{1,2}$ to $(\Alg_{(n)}(\DGCat))^{1-\Cat}$, the latter being the Morita $\infty$-category of $E_n$-monoidal DG categories (that is, the $\infty$-category underlying $\Alg^\circ_{(n)}(\DGCat)$).

\medskip

On the one hand, we can consider the restriction of $\bT_n$ to $\Bord_{1,2}$, which we denote by the same symbol:
$$
\bT_n
:
\Bord_{1,2}
\longto
\Alg^\circ_{(n)}(\DGCat)^{1-\Cat}.
$$
On the other hand, we can consider the \emph{lax} functor
$$
\bT'_n
:
\Bord_{1,2}
\longto
\Alg^\circ_{(n)}(\DGCat)
$$
defined as in Claim \ref{magic quasi-theorem}, that is, by the assignment $M^d \squigto \ICoh_0
\bigt{
(Y^{\partial(M \times D^{n+1-d})   })^\wedge_{Y^M}
}$ (for both objects and morphisms).
The excision theorem, Theorem \ref{thm:gluing along boundary}, proves that $\bT_n'$ is a genuine functor, rather than a lax one. Hence, $\bT_n'$ descends to a functor
$$
\Bord_{1,2}
\longto
\Alg^\circ_{(n)}(\DGCat)^{1-\Cat},
$$
denoted again by $\bT'_n$. We claim that $\bT'_n$ is symmetric monoidal: indeed, this is an immediate consequence of the tensor product theorem for $\ICoh_0$, see \cite[Proposition 3.12]{centerH}. The latter simply states that the natural arrow
$$
\ICoh_0((Z_1)^\wedge_{W_1})
\otimes
\ICoh_0((Z_2)^\wedge_{W_2})
\longto
\ICoh_0((Z_1 \times Z_2)^\wedge_{W_1 \times W_2})
$$
is an equivalence.

\sssec{}

Thus, we have two symmetric monoidal functors, $\bT_n$ and $\bT'_n$, carrying $\Bord_{1,2}$ into $(\Alg_{(n)}(\DGCat))^{1-\Cat}$.
Since they coincide when evaluated on $D^2$ and on $P^2$, they must coincide as functors (this is \cite[Proposition 3.3]{BZGN}). The assertion of Theorem \ref{main-thm-intro} follows.



\end{document}